\documentclass[a4paper, 10pt]{article}
\usepackage[left=4.1cm, right=4.1cm]{geometry}

\usepackage{packageList}
\usepackage{macros}

\usetikzlibrary{arrows}

\graphicspath{{./Figures/}}
\pgfplotsset{compat=1.11}
\newlength\fwidth

\title{Sharp inverse statements for kernel approximation: \\
Superconvergence and saturation} %

\author[1,2]{Tizian Wenzel \thanks{wenzel@math.lmu.de}}
\affil[1]{Deparment of Mathematics, Ludwig Maximilian University of Munich (Munich, Germany)}
\affil[2]{Munich Center for Machine Learning (Munich, Germany)}

\newcommand{\Ht}{\mathcal{H}_{\vartheta}}

\usepackage{comment}
\newif\iflong			
\longfalse

\newif\ifappendix
\appendixtrue

\usepackage{comment}

\usepackage[normalem]{ulem}			%
\usepackage{thm-restate}

\makeatletter
\newcommand*\bigcdot{\mathpalette\bigcdot@{.5}}
\newcommand*\bigcdot@[2]{\mathbin{\vcenter{\hbox{\scalebox{#2}{$\m@th#1\bullet$}}}}}
\makeatother

\begin{document}

\maketitle %

\begin{abstract}
This article establishes sharp inverse and saturation statements for kernel-based approximation using finitely smooth Sobolev kernels on bounded Lipschitz regions.
The analysis focuses on the superconvergence regime, for which direct statements have only recently been obtained.

The resulting theory yields a one-to-one correspondence between the smoothness of a target function -- quantified in terms of power spaces -- and the achievable approximation rates by kernel-based approximation. 
In this way, we extend existing results beyond the escaping-the-native-space regime and provide a unified characterization covering the full scale of admissible smoothness spaces.

\end{abstract}

\section{Introduction} \label{sec:introduction}

We consider kernel-based approximation \cite{wendland2005scattered, fasshauer2007meshfree, fasshauer2015kernel, steinwart2008support} 
and in particular interpolation of continuous functions $f \in \mathcal{C}(\Omega)$ on some bounded region $\Omega \subset \R^d$ for interpolation points $X \subset \Omega$:
\begin{align}
\label{eq:interpolation_conditions}
s_{f, X}(x_i) = f(x_i) \quad \forall x_i \in X.
\end{align}
For this, we consider strictly positive definite kernels $k: \Omega \times \Omega \rightarrow \R$ with corresponding reproducing kernel Hilbert space (RKHS) $\ns$,
which is also called the \emph{native space}.
In this setting,
the minimum-norm kernel interpolant $s_{f, X}$ can be written as
\begin{align}
\label{eq:interpolant_sum}
s_{f, X} = \sum_{j=1}^{|X|} \alpha_j k(\cdot, x_j),
\end{align}
where the $X$-dependent coefficient vector $(\alpha_j)_{j=1}^{|X|}$ is determined by the unique solution of the linear equation system arising from plugging Eq.~\eqref{eq:interpolant_sum} into Eq.~\eqref{eq:interpolation_conditions}.
If $f \in \ns$, then the kernel interpolant $s_{f, X}$ of Eq.~\eqref{eq:interpolation_conditions} can also be expressed as the orthogonal projection $\Pi_{V(X)}(f)$ of $f$ onto $V(X) := \Sp \{ k(\cdot, x_i), x_i \in X \} \subset \ns$, i.e.\
\begin{align}
\label{eq:interpolant_projection}
s_{f, X} = \Pi_{V(X)}(f).
\end{align}

For our analysis, we will be interested in finitely smooth kernels,
whose RKHS is norm-equivalent to a Sobolev space of smoothness $\tau > d/2$,
i.e.\ $\ns \asymp H^\tau(\Omega)$.
Such kernels are frequently used in practice,
as they cover the popular case of translational invariant kernels $k(x, z) = \Phi(x - z)$, 
for which the Fourier transform $\hat{\Phi}$ satisfies
\begin{align}
\label{eq:fourier_decay}
c_\Phi (1+\Vert \omega \Vert_2^2)^{-\tau} \leq \hat{\Phi}(\omega) \leq C_\Phi (1 + \Vert \omega \Vert_2^2)^{-\tau}
\end{align}
for some constants $c_\Phi, C_\Phi > 0$ and $\tau > d/2$.
In particular radial basis function (RBF) kernels like Matérn kernels or Wendland kernels satisfy this property.

Plenty of research has been conducted to bound the approximation error between functions $f \in \mathcal{C}(\Omega)$ (satisfying some additional smoothness assumptions) and the interpolant $s_{f, X}$,
see e.g.\ \cite{narcowich2006sobolev,arcangeli2007extension} using well distributed points,
\cite{wenzel2021novel,wenzel2023analysis} using adaptively (greedily) chosen points
or \cite{krieg2024random} using randomly chosen points.

For such bounds, one distinguishes two regimes, depending on the smoothness of the function $f$.
First, there is the \emph{escaping the native space regime} \cite{narcowich2004scattered,narcowich2006sobolev},
where the smoothness of $f$ is less than the smoothness of the RKHS $\ns$, i.e.\ it holds $f \notin \ns$.
In this case, error bounds read for example
\begin{align}
\label{eq:error_bound_escaping}
\Vert f - s_{f, X} \Vert_{L_2(\Omega)} \leq C h_X^{\vartheta \tau} \cdot \Vert f \Vert_{H^{\vartheta \tau}(\Omega)}
\end{align}
for $f \in H^{\vartheta \tau}(\Omega)$ with $\vartheta \in \left( \frac{d}{2\tau}, 1 \right]$ using quasi-uniform points $X \subset \Omega$.
Here, the quantity $h_X$ denotes the fill distance, see Eq.~\eqref{eq:fill_sep_dist},
and $H^{\vartheta \tau}(\Omega)$ refers to the ordinary Sobolev space of smoothness $\vartheta \tau > d/2$.

Second, there is the so-called \emph{superconvergence regime} \cite{schaback1999improved,schaback2018superconvergence,sloan2025doubling,karvonen2025general},
where the function $f$ possesses additional smoothness properties compared to $\ns$, 
i.e.\ $f$ is included in some particular subspaces.
These subspaces are so-called \emph{power spaces} $\calh_{\vartheta}(\Omega)$ (for $\vartheta > 1$) of the RKHS $\ns$ \cite{steinwart2012mercer}, 
see Eq.~\eqref{eq:power_spaces} for a precise definition.
The resulting convergence rates \cite{karvonen2025general} can be written as 
\begin{align}
\label{eq:error_bound_superconv}
\Vert f - s_{f, X} \Vert_{L_2(\Omega)} \leq C h_X^{\vartheta\tau} \cdot \Vert f \Vert_{\calh_\vartheta(\Omega)}
\end{align}
for functions $f \in \calh_{\vartheta}(\Omega)$ with $\vartheta \in [1, 2]$.
Note that for $\vartheta \in [0, 1]$, the norm equivalence $\ns \asymp H^\tau(\Omega)$ 
directly implies via interpolation the norm equivalence $\calh_\vartheta(\Omega) \asymp H^{\vartheta \tau}(\Omega)$ for $\vartheta \in [0, 1]$,
such that Eq.~\eqref{eq:error_bound_escaping} has actually the same form as Eq.~\eqref{eq:error_bound_superconv}.

Recently, also corresponding sharp \emph{inverse statements} for the escaping the native space regime have been developed \cite{wenzel2025sharp,avesani2025sobolev}.
These statements start with a given convergence rate as in Eq.~\eqref{eq:error_bound_escaping} and allow to conclude the exact smoothness of the approximated function $f$.
Like this, a one-to-one correspondence between smoothness and approximation rate was established.

This article will deal with sharp inverse statements for the superconvergence regime. 
This means, given an approximation rate as in Eq.~\eqref{eq:error_bound_superconv},
we will conclude the smoothness of the approximated function $f$, measured in terms of the power space parameter $\vartheta$.
Since the derived inverse statements are sharp as well,
the aforementioned one-to-one correspondence is extended from the escaping the native space regime to the superconvergence regime.
Deriving also saturation statements, which limit the maximal approximation rate,
a complete characterization of the rates of approximation in terms of the smoothness of a function is derived.
In order to work as general as possible,
we consider general kernel-based approximants, not necessarily interpolants.
An overview of references for the direct and inverse results in both regimes is provided in \Cref{table:overview_results}.
Our main result is precisely given as

\begin{restatable}{theorem}{mainresult}[Main result]
\label{th:main_result}
Consider a compact Lipschitz region $\Omega \subset \R^d$ and a continuous kernel $k$ such that $\ns \asymp H^\tau(\Omega)$ for some $\tau > d/2$.
Consider $f \in \mathcal{C}(\Omega)$ and the estimate 
\begin{align}
\label{eq:inverse_statement_assumption}
\Vert f - s_{f, X} \Vert_{L_2(\Omega)} \leq c_f h_{X}^\beta,
\end{align}
where $s_{f, X} \in \Sp \{ k(\cdot, x), x \in X \}$ is a kernel-based approximant. 
\begin{enumerate}[label=(\roman*)]
\item If Eq.~\eqref{eq:inverse_statement_assumption} holds for some $\beta \in (0, 2\tau - \frac{d}{2}]$ for one sequence $(X_n)_{n \in \N} \subset \Omega$ of point sets satisfying \Cref{ass:points},
then $f \in \mathcal{H}_{\vartheta}(\Omega)$ for all $\vartheta \in [0, \frac{\beta}{\tau})$.
\end{enumerate}
Additionally assuming \Cref{ass:method}, we have
\begin{enumerate}[label=(\roman*)]
\setcounter{enumi}{1}
\item If Eq.~\eqref{eq:inverse_statement_assumption} holds for some $\beta \in [2\tau - \frac{d}{2}, 2]$ 
for any quasi-uniform\footnotemark \ sequence $(X_n)_{n \in \N} \subset \Omega$, 
then $f \in \mathcal{H}_{\vartheta}(\Omega)$ for all $\vartheta \in [0, \frac{\beta}{\tau})$.
\item If Eq.~\eqref{eq:inverse_statement_assumption} holds for some $\beta > 2\tau$ 
for any quasi-uniform\footnotemark[\value{footnote}] sequence $(X_n)_{n \in \N} \subset \Omega$,
then $f = 0$.
\footnotetext{with uniformity constant $\rho = C_\Omega$ (see \Cref{prop:estimate_sep_fill_dist_Z})}

\end{enumerate}
\end{restatable}

\begin{table}[h!]
\centering
 \begin{tabular}{||c|| c  c||} 
 \hline
 & Escaping the native space & Superconvergence  \\ [0.5ex] 
 \hline\hline
Direct statements & \cite{narcowich2005sobolev,wendland2005approximate,narcowich2006sobolev,arcangeli2007extension} & \cite{schaback1999improved,schaback2018superconvergence,sloan2025doubling,karvonen2025general} \\ 
Inverse statements & \cite{schaback2002inverse,wenzel2025sharp,avesani2025sobolev} & (ours) \\ [1ex] 
 \hline
 \end{tabular}
 \caption{Overview of direct and inverse statements for both the escaping the native space and superconvergence regime.}
 \label{table:overview_results}
\end{table}

Such sharp inverse as well as saturation statements as provided in \Cref{th:main_result} enable a better understanding of kernel-based approximation methods using finitely smooth kernels.
Various potential applications are discussed later on in the article.

The article is structured as follows:
In \Cref{sec:background}, we review necessary background information on kernel-based approximation, provide further details on direct and inverse statements and review power spaces.
\Cref{sec:utility_statements} introduces and proves useful preparatory results for proving the main statement.
\Cref{sec:cont_superconv_inverse} then states and proves the main results on inverse statements for the superconvergence regime, 
and discusses the result and its proof.
Finally, \Cref{sec:conclusion} concludes the paper and provides an outlook.

\section{Background}
\label{sec:background}

In the following, we recall several statements which are necessary to derive the new results as well as helpful to understand the new results in the overall context.

\subsection{Kernel approximation}
\label{subsec:kernel_approx}

We are interested in real-valued strictly positive definite kernels $k$ on sets $\Omega$,
i.e.\ $k: \Omega \times \Omega \rightarrow \R$ is a symmetric function such that the kernel matrix $A_X$ is positive definite for any choice of pairwise distinct points $X \subset \Omega$ \cite{wendland2005scattered}.
For this, the kernel matrix is given by
\begin{align}
\label{eq:kernel_matrix}
A_X := (k(x_i, x_j))_{1 \leq i, j \leq |X|} \in \R^{|X| \times |X|}.
\end{align}
For every strictly positive definite kernel,
there is a unique associated space of functions, the so-called \emph{reproducing kernel Hilbert space} $\ns$ (RKHS),
which is also frequently called \emph{native space}.
This space is characterized by the following two properties
\begin{align}
\label{eq:reproducing_property}
\begin{aligned}
&k(\cdot, x) \in \ns \qquad ~ \qquad \forall x \in \Omega \\
&f(x) = \langle f, k(\cdot, x) \rangle_{\ns} \quad \forall x \in \Omega, ~ \forall f \in \ns \hfill.
\end{aligned}
\end{align}
The second property is called the reproducing property.

We will be mostly interested in kernels,
for which the RKHS $\ns$ is norm-equivalent to a Sobolev space $H^\tau(\Omega)$ with $\tau > d/2$.
This means that the spaces coincide as sets and the norms are equivalent.
This is the case for frequently used kernels such as the family of Matérn kernels or Wendland kernels
on compact Lipschitz regions.
The main results of this paper will be formulated for such kernels,
which is thus collected in the following assumption.
\begin{assumption}
\label{ass:kernel_region}
Let $\Omega \subset \R^d$ be a compact Lipschitz region $\Omega \subset \mathbb{R}^d$ and let $k: \Omega \times \Omega \rightarrow \R$ be a continuous kernel such that $\ns \asymp H^\tau(\Omega)$ for some $\tau > d/2$. 
\end{assumption}
For use later, we already remark that a Lipschitz region implies an interior cone condition \cite[Lemma 1.5]{krieg2024random}.
This will be exploited later on.
The interior cone conditions reads as follows.
\begin{definition}
A set $\Omega \subset \R^d$ is said to satisfy an interior cone condition, if there exists an angle $\alpha \in (0, \pi/2)$ and a radius $r>0$ such that for every $x \in \Omega$ a unit vector $\xi(x)$ exists such that the cone 
\begin{align}
\label{eq:definition_cone}
C(x, \xi(x), \alpha, r) := \{ x + \lambda y: y \in \R^d, \Vert y \Vert_2 = 1, y^\top \xi(x) \geq \cos(\alpha), \lambda \in [0, r] \}
\end{align}
is contained in $\Omega$.
\end{definition}
In order to quantify approximation errors, 
some geometric quantities on the distribution of the points $X \subset \Omega$ are necessary.
For a given set of points $X \subset \Omega$ within some region $\Omega$, 
we recall the definition of the seperation distance $q_X$ and the fill distance $h_X$:
\begin{align}
\label{eq:fill_sep_dist}
\begin{aligned}
q_{X} :=& \frac{1}{2} \min_{x_i \neq x_j \in X} \Vert x_i - x_j \Vert_2, \\
h_X := h_{X, \Omega} :=& \sup_{x \in \Omega} \min_{x_j \in X} \Vert x - x_j \Vert_2.
\end{aligned}
\end{align}
For a bounded region that satisfies an interior cone condition, 
there is a constant $c_\Omega > 0$ such that it holds
\begin{align}
\label{eq:bound_fill_sep_dist}
\frac{h_{X}}{q_X} \geq c_\Omega
\end{align}
for any set $X \subset \Omega$ of least two points.
This follows from volume comparison arguments \cite[Section 14.1]{wendland2005scattered}.
For the opposite bound, one considers the so-called uniformity constant defined as
\begin{align}
\label{eq:uniformity_constant}
\rho_X := \frac{h_X}{q_X}.
\end{align}
If the uniformity constant is uniformly bounded for a sequence of sets $(X_n)_{n \in \N} \subset \Omega$,
the sequence is called quasi-uniform.
In the following, we will consider such quasi-uniform sequences of sets of points, 
and thus state the following assumption:
\begin{assumption}
\label{ass:points}
Let $(X_n)_{n \in \N} \subset \Omega$ be a nested sequence of quasi-uniformly distributed point sets with geometrically decaying fill distance,
i.e.\ it holds
\begin{align}
\label{eq:assumption_decay_fill_dist}
c_0' a^n \leq q_{X_n} \leq h_{X_n} \leq c_0'' a^n
\end{align}
for constants $c'_0, c_0'' > 0$ and $a \in (0, 1)$.
\end{assumption}
Standard error estimates are often given in $L_p(\Omega)$ norms for $1 \leq p \leq \infty$ \cite{wendland2005scattered},
and in this manuscript we focus on $L_2(\Omega)$ error estimates.
For functions $f \in \ns$, 
it holds the following standard error estimate for the kernel interpolant $s_{f, X}$:
\begin{align}
\label{eq:standard_bound}
\Vert f - s_{f, X} \Vert_{L_2(\Omega)} \leq C h_{X}^{\tau} \cdot \Vert f - s_{f, X} \Vert_{\ns}
\end{align}
In the superconvergence regime \cite{karvonen2025general},
where $f$ is included in some particular subspaces of $\ns$,
the RKHS residual term $\Vert f - s_{f, X} \Vert_{\ns}$ is decaying and can be further estimated.
On the contrary, if $f$ is not included in the RKHS $\ns$,
which is the so-called escaping (the native space) regime, 
the convergence rate is actually smaller.
Details on these regimes are provided in \Cref{subsec:direct_inverse} and \Cref{subsec:power_spaces_superconv}.

In order to assess the stability of kernel-based approximation method,
the kernel matrix $A_X$ of Eq.~\eqref{eq:kernel_matrix} has been thoroughly analyzed \cite{wendland2005scattered,diederichs2019improved,wenzel2025spectral}.
We will recall a stability bound on its smallest eigenvalue $\lambda_{\min}(A_X)$:
Under \Cref{ass:kernel_region},
i.e.\ Sobolev kernels on compact Lipschitz regions,
there exists a constant $c_0 > 0$ such that for any set of pairwise distinct points $X \subset \Omega$ it holds \cite[Theorem 5]{wenzel2025sharp}
\begin{align}
\label{eq:estimate_lambda_min}
\begin{aligned}
\lambda_{\min}(A_{X}) &\geq c_0 q_{X}^{2\tau - d}, \\
\Rightarrow \quad \Vert A_X^{-1} \Vert_{2,2} &\leq c_0^{-1} q_X^{d-2\tau}.
\end{aligned}
\end{align}

\subsection{Direct and inverse statements (escaping regime)}
\label{subsec:direct_inverse}

This subsection recalls direct and inverse statments for the \emph{escaping regime}.
This regime covers all functions, which are outside the RKHS.
The opposite \emph{superconvergence regime},
which covers functions that are in particular subspaces of the RKHS,
is treated in \Cref{subsec:power_spaces_superconv}.

First we start by recalling a direct statement, 
which is a special case of \cite[Theorem 4.2]{narcowich2006sobolev} with the improvement $\beta > d/2$ instead of $\lfloor \beta \rfloor > d/2$ according to \cite[Theorem 4.1]{arcangeli2007extension}.
\begin{theorem}
\label{th:error_estimate}
Consider a RBF kernel $k$ that satisfies Eq.\ \eqref{eq:fourier_decay} for some $\tau > d/2$ and a bounded region $\Omega \subset \R^d$ with Lipschitz boundary. 
For some $\beta$ with $d/2 < \beta \leq \tau$ let $f \in H^\beta(\Omega) \supseteq H^\tau(\Omega)$.
Then the kernel interpolant $s_{f, X}$ satisfies
\begin{align*}
\Vert f - s_{f,X} \Vert_{L_2(\Omega)} \leq C h_{X}^\beta \rho_{X}^{\tau-\beta} \Vert f \Vert_{H^\beta(\Omega)}.
\end{align*}
\end{theorem}
Note that this result can also be formulated in terms of power spaces,
because they are norm-equivalent to the Sobolev spaces, see e.g.\ \cite{wenzel2025sharp}.
This means, 
the right hand side norm $\Vert f \Vert_{H^\beta(\Omega)}$ can be replaced by the equivalent power space norm $\Vert f \Vert_{\mathcal{H}_{\beta/\tau}(\Omega)}$,
as discussed in \Cref{subsec:power_spaces_superconv}.

The direct statement of \Cref{th:error_estimate} was complemented by a corresponding sharp inverse statement \cite{wenzel2025sharp}.
Here we recall a slightly extended version of this inverse statement (with weakened assumptions), taken from \cite{avesani2025sobolev}.
\begin{theorem}[Theorem 3.5 of \cite{avesani2025sobolev}]
\label{thm:l2-inverse-statement_strengthened}
Under \Cref{ass:kernel_region} and \Cref{ass:points},
let $f \in L_2(\Omega)$ and assume there exists a sequence of point based approximants, i.e.\ $(s_{f, X_n})_{n \in \N} \subset \ns$ with $s_{f, X_n} \in \Sp \{ k(\cdot, x), x \in X_n \}$ such that
\begin{align}
\label{eq:l2-inverse-statement_strengthened}
\Vert f - s_{f, X_n} \Vert_{L_2(\Omega)} \leq c_f h_{X_n}^\beta
\end{align}
holds for some $c_f >0$ and $\beta \in (0, \tau]$.
Then $f \in H^{\beta'}(\Omega)$ for all $\beta' \in (0, \beta)$. 
\end{theorem}
We recall,
that the direct statement of \Cref{th:error_estimate} and the corresponding inverse statement of \Cref{thm:l2-inverse-statement_strengthened} are sharp,
and thus establish a one-to-one correspondence between smoothness and approximation rate for the escaping regime \cite{wenzel2025sharp}.

A key tool for deriving such an inverse statement are Bernstein inequalities,
which allow to bound strong norms in terms of weak norms.
The proof of \Cref{thm:l2-inverse-statement_strengthened} essentially relied on such a Bernstein inequality,
 which was recently derived in \cite{zhengjie2025inverse} and slightly generalized in \cite[Theorem 3.4]{avesani2025sobolev}.
It states that it holds
\begin{align}
\label{eq:bernstein_old}
\Vert u \Vert_{H^s(\Omega)} \leq C q_{X}^{-s} \Vert u \Vert_{L_2(\Omega)}
\end{align}
for a constant $C = C_{d, k, \tau, \Omega} > 0$, $s \in [0, \tau]$
and for all 
point sets $X \subset \Omega$ and all trial functions $u \in \Sp \{k(\cdot, x), x \in X \}$.
Unfortunately,
the Bernstein inequality in \cite{zhengjie2025inverse} as well as in \cite{avesani2025sobolev} is formulated for RBF kernels.
In \Cref{prop:escaping_bernstein}, 
we prove that such a Bernstein inequality actually holds for any Sobolev kernel as defined in \Cref{ass:kernel_region},
i.e.\ assuming a radial kernel is not necessary.

\subsection{Mercer's theorem, power spaces and superconvergence}
\label{subsec:power_spaces_superconv}

In order to introduce power spaces and stating the superconvergence results from \cite{karvonen2025general},
we start by introducing the kernel integral operator \cite{steinwart2012mercer}:
For a given continuous strictly positive definite kernel $k: \Omega \times \Omega \rightarrow \R$ on a bounded Lipschitz region $\Omega \subset \R^d$, 
the associated kernel integral operator $T := T_k: L_2(\Omega) \rightarrow L_2(\Omega)$ is given as
\begin{align*}
Tv(x) := \int_{\Omega} k(x, z) v(z) ~ \mathrm{d}z, \qquad v \in L_2(\Omega), ~ x \in \Omega.
\end{align*}
This operator $T$ is self-adjoint and positive in $L_2(\Omega)$ with ordered eigenvalues $\lambda_1 \geq \lambda_2 \geq ... > 0$ and corresponding eigenfunctions $( \varphi_j)_{j \in \N}$,
which form an orthonormal basis (ONB) of $L_2(\Omega)$.
Furthermore it holds $\Vert \varphi_j \Vert_{\ns} = 1/\sqrt{\lambda_j}$.

Moreover, one can show that $T$ viewed as an operator $L_2(\Omega) \rightarrow \ns$ is the adjoint of the embedding operator $\ns \hookrightarrow L_2(\Omega)$ \cite[Proposition 10.28]{wendland2005scattered}, 
such that the following important identity holds:
\begin{align}
\label{eq:important_identity}
\langle f, Tv \rangle_{\ns} = \langle f, v \rangle_{L_2(\Omega)} \qquad \forall v \in L_2(\Omega), ~ f \in \ns.
\end{align}
Note that \cite[Example 20]{karvonen2025general} extended this identity to functions $v \in L_p(\Omega)$, $p \geq 1$,
however for us the $p=2$ identity is sufficient.
Using the ONB $(\varphi_j)_{j \in \N}$,
we can introduce the scale of power spaces for $\vartheta \in [0, \infty)$ as
\begin{align}
\label{eq:power_spaces}
\Ht(\Omega) := \{ f \in L_2(\Omega) ~:~ \sum_{j=1}^\infty \frac{|\langle f, \varphi_j \rangle_{L_2(\Omega)}|^2}{\lambda_j^\vartheta}  < \infty \} =\left\{
\begin{array}{ll}
L_2(\Omega) & \vartheta = 0 \\
\ns & \vartheta = 1 \\
TL_2(\Omega) & \vartheta = 2 \\
\end{array}
\right.
\end{align}
with inner products given as
\begin{align}
\label{eq:dot_product}
\langle f, g \rangle_{\Ht(\Omega)} = \sum_{j = 1}^\infty \frac{\langle f, \varphi_j \rangle_{L_2(\Omega)} \langle g, \varphi_j \rangle_{L_2(\Omega)}}{\lambda_j^\vartheta}.
\end{align}
These spaces are complete, and for $\vartheta$ sufficiently large these spaces are RKHS \cite{steinwart2012mercer}.
For $\vartheta = 1$ we obtain the standard RKHS $\ns$, i.e.\ $\mathcal{H}_{\vartheta=1}(\Omega) = \ns$.
It is important to note that the index $k$ of $\ns$ refers to the kernel $k$, while the index $\vartheta \geq 0$ refers to a scalar value.
The scale of spaces $\mathcal{H}_\vartheta(\Omega)$ in visualized in \Cref{fig:visualization},
while a more complete picture including $TL_p(\Omega)$ spaces is provided in \cite{karvonen2025general}.
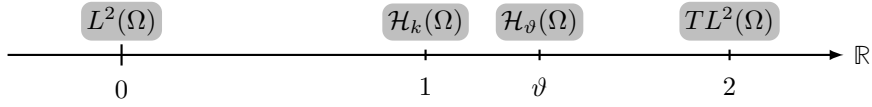
\begin{figure}[t]
\setlength\fwidth{.4\textwidth}
\begin{center}
\begin{tikzpicture}[>=latex, thick]
\draw[->] (0,0) -- (11cm,0) node [right] {$\R$};

\draw (1.5,-4pt) -- (1.5,4pt) node[below=10pt]{0};
\draw (1.5,5pt) node[above, align=center, fill=lightgray, rounded corners, inner sep=2pt]{$L^2(\Omega)$};

\draw (7,-2pt) -- (7,2pt);
\draw (7, 5pt) node[above, align=center, fill=lightgray, rounded corners, inner sep=2pt]{$\mathcal{H}_\vartheta(\Omega)$} node[below=10pt]{$\vartheta$};

\draw (5.5,-2pt) -- (5.5,2pt);
\draw (5.5, 5pt) node[above, align=center, fill=lightgray, rounded corners, inner sep=2pt]{$\ns$} node[below=10pt]{1};
\draw (9.5,-2pt) -- (9.5,2pt);
\draw (9.5, 5pt) node[above, align=center, fill=lightgray, rounded corners, inner sep=2pt]{$TL^2(\Omega)$} node[below=10pt]{2};
\end{tikzpicture}
\end{center}
\caption{Visualization of the scale of power spaces with the special cases $L_2(\Omega)$ (for $\vartheta = 0$), $\ns$ (for $\vartheta = 1$) and $TL_2(\Omega)$ (for $\vartheta = 2$).
This article mainly considers the superconvergence case, i.e.\ $\vartheta \in [1, 2]$.}
\label{fig:visualization}
\end{figure}

For the Sobolev kernels of interest in this manuscript, 
i.e.\ under \Cref{ass:kernel_region},
the eigenvalues $\lambda_j$ decay asymptotically as $j^{-2\tau/d}$ \cite{santin2016approximation}, 
i.e.\ there are constants $c, C > 0$ such that
\begin{align}
\label{eq:asympt_eigvals}
c j^{-2\tau/d} \leq \lambda_j \leq C j^{-2\tau/d}.
\end{align}
Still under \Cref{ass:kernel_region}, it is worth to note that for $\vartheta > \frac{d}{2\tau} =: \vartheta_\text{inf}$, 
the spaces $\mathcal{H}_\vartheta(\Omega)$ are again RKHS with reproducing kernel given by the \emph{power kernel}
\begin{align}
\label{eq:power_kernel}
k^{(\vartheta)}(x, z) = \sum_{n=1}^\infty \lambda_n^{\vartheta} \varphi_n(x) \varphi_n(z),
\end{align}
Especially $\vartheta_\text{inf}$ is a lower bound for possible $\vartheta$ values.
For more details, see e.g.\ \cite{steinwart2012mercer}.

Historically, 
the first superconvergence statements focussed on functions in the image $TL_2(\Omega)$ of the integral operator \cite{schaback1999improved,schaback2018superconvergence,sloan2025doubling}.
Recently, \cite{karvonen2025general} provided general superconvergence statements,
including in particular the spaces $\mathcal{H}_\vartheta(\Omega)$ for $\vartheta \in [1, 2]$.
These general superconvergence statements can be proven in the particular setting of power spaces via the Hölder inequality
\begin{align}
\label{eq:power_space_hoelder}
\left| \langle f, g \rangle_{\mathcal{H}_1(\Omega)} \right| \leq \Vert f \Vert_{\Ht(\Omega)} \cdot \Vert g \Vert_{\mathcal{H}_{2-\vartheta}(\Omega)},
\end{align}
for $f \in \Ht(\Omega), g \in \mathcal{H}_{2-\vartheta}(\Omega)$, $\vartheta \in [1, 2]$
and the interpolation inequality 
\begin{align}
\label{eq:power_space_interpol}
\Vert f \Vert_{\mathcal{H}_{\vartheta}(\Omega)} &\leq \Vert f \Vert_{L_2(\Omega)}^{1 - \vartheta} \cdot \Vert f \Vert_{\ns}^{\vartheta}, \qquad f \in \ns.
\end{align}
Eq.~\eqref{eq:power_space_hoelder} and Eq.~\eqref{eq:power_space_interpol} combined directly 
yield the direct statement of Eq.~\eqref{eq:error_bound_superconv} from \cite{karvonen2025general},
for which we provide a brief proof for convenience:
\begin{theorem}
\label{cor:direct_statement}
Consider a kernel $k$ that satisfies Eq.~\eqref{eq:fourier_decay} for some $\tau > d/2$ and a bounded Lipschitz region $\Omega \subset \R^d$. 
Consider $f \in \mathcal{H}_\vartheta(\Omega)$ for $\vartheta \in [1, 2]$.
Then it holds
\begin{align*}
\Vert f - s_{f, X} \Vert_{L_2(\Omega)} \leq C h_{X}^{\vartheta \tau} \cdot \Vert f \Vert_{\mathcal{H}_{\vartheta}(\Omega)}
\end{align*}
for any $X \subset \Omega$.
\end{theorem}
\begin{proof}
We start with the standard error bound of Eq.~\eqref{eq:standard_bound}
\begin{align*}
\Vert f - s_{f, X} \Vert_{L_2(\Omega)} \leq C h_{X}^{\tau} \Vert f - s_{f, X} \Vert_{\ns},
\end{align*}
and further estimate the $\Vert f - s_{f, X} \Vert_{\ns}$ factor by using Eq.~\eqref{eq:power_space_hoelder} and Eq.~\eqref{eq:power_space_interpol}.
Note that the kernel interpolant can be expressed as an orthogonal projection, see Eq.~\eqref{eq:interpolant_projection}:
\begin{align*}
\Vert f - s_{f, X} \Vert_{\ns}^2 
&= \langle f, f - s_{f, X} \rangle_{\ns} \\ 
&\leq \Vert f \Vert_{\mathcal{H}_\vartheta(\Omega)} \cdot \Vert f - s_{f, X_n} \Vert_{\mathcal{H}_{2-\vartheta}(\Omega)} \\
&\leq \Vert f \Vert_{\mathcal{H}_\vartheta(\Omega)} \cdot \Vert f - s_{f, X} \Vert_{L_2(\Omega)}^{\vartheta - 1} \cdot \Vert f - s_{f, X} \Vert_{\ns}^{2 - \vartheta} \\
\Leftrightarrow \Vert f - s_{f, X} \Vert_{\ns} &\leq \Vert f \Vert_{\mathcal{H}_\vartheta(\Omega)}^{\frac{1}{\vartheta}} \cdot \Vert f - s_{f, X} \Vert_{L_2(\Omega)}^{\frac{\vartheta - 1}{\vartheta}}.
\end{align*}
Plugging both together and rearranging yields the result.
\end{proof}

The power spaces $\calh_\vartheta(\Omega)$ can also be described as the closure of images of the integral operator under $L_p(\Omega)$ spaces.
For proving our main result \Cref{th:main_result}, we need in particular the following statement,
taken from \cite[Theorem 29]{karvonen2025general}
\begin{align}
\label{eq:power_space_via_closure}
\overline{TL_2(\Omega)}^{\Vert \cdot \Vert_{\calh_\vartheta(\Omega)}} = \calh_\vartheta(\Omega), \qquad \forall \vartheta \in [0, 2].
\end{align}

\section{Utility statements}
\label{sec:utility_statements}

This section derives (technical) utility results, which are important in their own,
and which are crucial for proving the main result.
Readers only interested in the main statement and its proof may proceed with \Cref{sec:cont_superconv_inverse} and come back later.
Note that in the following, constants within proofs may change from line to line.

\subsection{Generalized reproducing property}

A key step for the proof of the inverse statement \Cref{th:main_result} will be the following theorem,
which we will call a \emph{generalized reproducing property}, because it generalizes the standard reproducing property:

\begin{prop}[Generalized reproducing property]
\label{prop:generalized_repr_prop}
Under \Cref{ass:kernel_region},
let $\vartheta_1, \vartheta_2 \in \R_{\geq 0}$ such that $2\vartheta_1 - \vartheta_2 > \frac{d}{2\tau}$.
Then it holds for all $x, z \in \Omega$:
\begin{align*}
\langle k^{(\vartheta_1)}(\cdot, x), k^{(\vartheta_1)}(\cdot, z) \rangle_{\mathcal{H}_{\vartheta_2}(\Omega)} = k^{(2\vartheta_1 - \vartheta_2)}(x, z)
\end{align*}
\end{prop}

\begin{proof}
This is a straightforward calculation,
where convergence is ensured due to $2\vartheta_1 - \vartheta_2 > \frac{d}{2\tau}$ (see the discussion around Eq.~\eqref{eq:power_kernel}):
\begin{align*}
\langle k^{(\vartheta_1)}(\cdot, x), k^{(\vartheta_1)}(\cdot, z) \rangle_{\mathcal{H}_{\vartheta_2}(\Omega)} 
&\equiv \sum_{j=1}^\infty \frac{\langle k^{(\vartheta_1)}(\cdot, x), \varphi_j \rangle_{L_2(\Omega)} \langle k^{(\vartheta_1)}(\cdot, z), \varphi_j \rangle_{L_2(\Omega)}}{\lambda_j^{\vartheta_2}} \\
&= \sum_{j=1}^\infty \frac{1}{\lambda_j^{\vartheta_2}} \cdot \lambda_j^{\vartheta_1} \varphi_j(x) \cdot \lambda_j^{\vartheta_1} \varphi_j(z) \\
&= \sum_{j=1}^\infty \lambda_j^{2\vartheta_1-\vartheta_2} \cdot \varphi_j(x) \varphi_j(z) \\
&\equiv k^{(2\vartheta_1 - \vartheta_2)}(x, z).
\end{align*}
\end{proof}

Note that the kernels $k^{(\vartheta_1)}(\cdot, x)$ are only defined for $\vartheta_1 \geq \frac{d}{2\tau}$, 
however the expression on the left hand side makes also sense for further values of $\vartheta_1, \vartheta_2$ as long as $2\vartheta_1 - \vartheta_2 > \frac{d}{2\tau}$.
For $\vartheta_1 = \vartheta_2 = 1$,
one reobtains the standard reproducing property from Eq.~\eqref{eq:reproducing_property}, i.e.\ $\langle k(\cdot, x), k(\cdot, z) \rangle_{\ns} = k(x, z)$.
Furthermore, 
note that \Cref{prop:generalized_repr_prop} works because according to \cite[Theorem 5.3]{steinwart2012mercer} we have 
\begin{align}
\label{eq:kernel_inclusion}
\begin{aligned}
k(\cdot, x) \in \Ht(\Omega) \quad \forall x \in \Omega \quad 
&\Leftrightarrow \quad k^{(2 - \vartheta)} ~ \text{is well defined} \\
&\Leftrightarrow \quad \vartheta \in \left[0, 2-\frac{d}{2\tau} \right).
\end{aligned}
\end{align}
The statement of \Cref{prop:generalized_repr_prop} does not crucially rely on \Cref{ass:kernel_region},
in fact it works as soon as Mercer's theorem is applicable.

\subsection{Bernstein inequalities}

In this subsection, 
we first slightly extend a Bernstein inequality from the escaping the native space regime:
The Bernstein inequalities in this regime from the literature \cite{zhengjie2025inverse,avesani2025sobolev} 
were actually only stated and proven for RBF kernels.
By carefully checking the proofs of the literature, we see that the Bernstein inequality actually also holds for non-radial kernels.
Second, 
we extend that Bernstein inequality from the escaping the native space regime to the superconvergence regime.
This means, we are interested in bounding $\Vert \cdot \Vert_{\calh_\vartheta(\Omega)}$ norms for $\vartheta > 1$ in terms of weaker norms.
For this, we state and prove the most basic form using the $\Vert \cdot \Vert_{L_2(\Omega)}$ norm, 
and remark that a generalization to further intermediate norms on the left hand side is straightforward by applying interpolation theory.

\begin{prop}[Bernstein inequality in the escaping regime]
\label{prop:escaping_bernstein}
Under \Cref{ass:kernel_region},
let $\vartheta \in [0, 1]$.
Then there exists a constant $C = C_{d, k, \tau, \Omega} > 0$ such that the following Bernstein inverse inequality holds
\begin{align}
\label{eq:escaping_bernstein}
\Vert u \Vert_{\calh_{\vartheta}(\Omega)} 
\asymp \Vert u \Vert_{H^{\vartheta \tau}(\Omega)}
\leq C q_{X}^{-\vartheta\tau} \Vert u \Vert_{L_2(\Omega)}
\end{align}
for all point sets $X \subset \Omega$ and all trial functions $u \in \Sp \{k(\cdot, x), x \in X \}$.
\end{prop}

The proof of \Cref{prop:escaping_bernstein} is just a careful checking of the proofs given in the literature, and does not require any new major ideas.
Thus we provide the proof in \Cref{subsec:bernstein_ineq_proof}.

\begin{prop}[Bernstein inequality in the superconvergence regime]
\label{prop:superconv_bernstein}
Under \Cref{ass:kernel_region},
let $\vartheta \in [1, 2 - \frac{d}{2\tau})$.
Then there exists a constant $C = C_{d, k, \tau, \Omega} > 0$ such that the following Bernstein inverse inequality holds
\begin{align}
\label{eq:superconv_bernstein}
\Vert u \Vert_{\calh_{\vartheta}(\Omega)} \leq C q_{X}^{-\vartheta\tau} \Vert u \Vert_{L_2(\Omega)}
\end{align}
for all point sets $X \subset \Omega$ and all trial functions $u \in \Sp \{k(\cdot, x), x \in X \}$.
\end{prop}
\begin{proof}
Consider $u = \sum_{j=1}^M \alpha_j k(\cdot, x_j)$.
By using the generalized reproducing property from \Cref{prop:generalized_repr_prop} we have
\begin{align*}
\Vert u \Vert_{\mathcal{H}_{\vartheta}(\Omega)}^2 
=& \sum_{i,j=1}^M \alpha_i \alpha_j \langle k(\cdot, x_i), k(\cdot, x_j) \rangle_{\mathcal{H}_{\vartheta}(\Omega)} 
=~ \sum_{i,j=1}^M \alpha_i \alpha_j k^{(2-\vartheta)}(x_i, x_j) \\
=:&~ \alpha^\top A^{(2-\vartheta)} \alpha.
\end{align*}
Here we introduced $A^{(2-\vartheta)}$ for the kernel matrix using the kernel $k^{(2-\vartheta)}$.
We proceed by inserting $A^{1/2} A^{-1/2}$, where $A$ denotes the kernel matrix using the kernel $k$:
\begin{align*}
=&~ \alpha^\top A^{1/2} A^{-1/2} A^{(2-\vartheta)} A^{-1/2} A^{1/2} \alpha \\
\leq&~ \Vert A^{1/2} \alpha \Vert^2 \cdot \Vert A^{-1/2} A^{(2-\vartheta)} A^{-1/2} \Vert_{2, 2} \\
\leq&~ (\alpha^\top A \alpha) \cdot \lambda_{\max}(A^{-1/2} A^{(2-\vartheta)} A^{-1/2}) \\
=&~ \Vert u \Vert_{\ns}^2 \cdot \max_{0 \neq \alpha \in \R^n} \frac{\alpha^\top A^{-1/2} A^{(2-\vartheta')} A^{-1/2} \alpha}{\alpha^\top \alpha} \\
=&~ \Vert u \Vert_{\ns}^2 \cdot \max_{0 \neq \beta \in \R^n} \frac{\beta^\top A^{(2-\vartheta')} \beta}{\beta^\top A \beta}. 
\end{align*}
For the first factor, we can directly apply the Bernstein inequality Eq.~\eqref{eq:escaping_bernstein}, giving the upper bound $\left( C q_X^{-\tau} \Vert u \Vert_{L_2(\Omega)} \right)^2$.
For the second factor we make use of the spectral alignment result \cite[Corollary 4.6]{wenzel2025spectral},
giving the upper bound $C q_X^{(1-\vartheta)2\tau}$.
Simplifying the expression yields the result.
\end{proof}
Note that \Cref{prop:superconv_bernstein} extends \Cref{prop:escaping_bernstein} in a natural way to the superconvergence norms,
and the restriction $\vartheta < 2 - \frac{d}{2\tau}$ is necessary due to Eq.~\eqref{eq:kernel_inclusion}.
Both statements together provide Bernstein inequalities for the full range $\vartheta \in [0, 2-\frac{d}{2\tau})$,
which we formulate as the following theorem:
\begin{theorem}[Bernstein inequality]
\label{th:bernstein}
Under \Cref{ass:kernel_region},
let $\vartheta \in [0, 2 - \frac{d}{2\tau})$.
Then there exists a constant $C = C_{d, k, \tau, \Omega} > 0$ such that the following Bernstein inverse inequality holds
\begin{align}
\label{eq:superconv_bernstein}
\Vert u \Vert_{\calh_{\vartheta}(\Omega)} \leq C q_{X}^{-\vartheta\tau} \Vert u \Vert_{L_2(\Omega)}
\end{align}
for all point sets $X \subset \Omega$ and all trial functions $u \in \Sp \{k(\cdot, x), x \in X \}$.
\end{theorem}

\subsection{Construction of density functions}
\label{subsec:construction_density_func}

In this subsection,
we provide a construction of density functions $v_n$,
such that $Tv_n$ approximates the function of interest $f$.
We first start with motivating this approach,
then giving the precise definition and deriving necessary properties.

\subsubsection{Motivation}

The motivation for \Cref{subsec:construction_density_func} is to establish the mathematical foundations of the proof technique for proving the main result in the range $\vartheta \in (2-\frac{d}{2\tau}, 2]$.
In this range, it holds $k(\cdot, x) \notin \calh_\vartheta(\Omega)$ by Eq.~\eqref{eq:kernel_inclusion}, 
thus one can no longer rely on the same proof strategy as for the case $\vartheta \in (0, 2-\frac{d}{2\tau})$,
which uses $s_{f, X_n} = \sum_{j=1}^{|X_n|} \alpha_j k(\cdot, x_j)$ as ansatz functions.

Thus the idea is to use the characterization of $\calh_\vartheta(\Omega)$ as the $\Vert \cdot \Vert_{\calh_\vartheta(\Omega)}$ norm closure of $TL_2(\Omega)$, see Eq.~\eqref{eq:power_space_via_closure}.
Therefore this subsection will introduce a construction of density functions $(v_n)_{n \in \N} \subset L_2(\Omega)$,
such that $Tv_n \rightarrow f$.
These density functions $v_n$ will be defined based on the coefficients $\alpha_j$ of $s_{f, Z} = \sum_{j=1}^{|Z|} \alpha_j k(\cdot, z_j)$ -- or viewed the other way round:
The coefficients $(\alpha_j)_{j=1}^{|Z|}$ constitute a discretization of the density function.
The motivation for this is given by 
\begin{align}
\label{eq:motivation_dirac}
k(\cdot, x) = \lim_{\varepsilon \rightarrow 0} T\delta_{x, \varepsilon} = \lim_{\varepsilon \rightarrow 0} \int_\Omega k(\cdot, z) \delta_{x, \varepsilon}(z) ~ \mathrm{d}z,
\end{align}
where $(\delta_{x, \varepsilon})_{\varepsilon > 0}$ is an approximation of the dirac delta for $\varepsilon \rightarrow 0$.
In particular we will make use of $\delta_{x, \varepsilon}(z) = \frac{1}{\varepsilon^d} \cdot \mathbb{1}_{x+[0, \varepsilon]^d}(z)$,
which allows to nicely decompose the region $\Omega$ and approximate it using small cubes.

\subsubsection{Assumption on kernel-based approximation method}

In order to work as generally as possible, 
we introduce the following assumption on the kernel-based approximation method:

\begin{assumption}
\label{ass:method}
Let $k: \Omega \times \Omega \rightarrow \R$ be a continuous strictly positive definite kernel on a compact Lipschitz region $\Omega \subset \R^d$.
We consider a point-based kernel approximation method,
i.e.\ for any pairwise distinct $X \subset \Omega$ we obtain an approximant
\begin{align*}
s_{f, X} =: \sum_{x_i \in X} \alpha_{f, X; x_i} k(\cdot, x_i)
\end{align*}
with coefficients $(\alpha_{f, X; x_i})_{i=1, ..., |X|} \subset \R^{|X|}$.
For any $f \in \mathcal{C}(\Omega)$,
we assume the coefficients $(\alpha_{f, X; x_i})_{i=1, ..., |X|}$ to depend continuously on $X$.
\end{assumption}

Next we will see,
that this assumption is frequently satisfied,
e.g.\ for kernel interpolation or regularized kernel interpolation.

\begin{prop}
\label{prop:method}
Let $k: \Omega \times \Omega \rightarrow \R$ be a continuous strictly positive definite kernel on a compact Lipschitz region $\Omega \subset \R^d$. 
Then both kernel interpolation as well a regularized kernel interpolation satisfy \Cref{ass:method}.
\end{prop}

\begin{proof}
For any pairwise distinct $X \subset \Omega$,
the coefficients $\alpha_X := (\alpha_{f, X; x_i})_{i=1, ..., |X|} \in \R^{|X|}$ for kernel interpolation and regularized kernel interpolation are determined by the solution of the linear equation system 
\begin{align*}
(A_X + \lambda \mathbb{1}) \alpha_X &= f(X), \\
\Leftrightarrow ~~ \alpha_X &= (A_X + \lambda \mathbb{1})^{-1} f(X).
\end{align*}
Here $A_X$ is the kernel matrix as defined in Eq.~\eqref{eq:kernel_matrix} and $f(X) := (f(x_i))_{i=1, ..., |X|} \in \R^{|X|}$.
Furthermore it holds $\lambda = 0$ for kernel interpolation and $\lambda > 0$ for regularized kernel interpolation.

Since the kernel $k$ is assumed to be continuous on $\Omega \times \Omega$ and $f$ is continuous on $\Omega$,
both $X \mapsto (k(X, X) + \lambda)^{-1}$ and $X \mapsto f(X)$ are continuous.
Thus it follows that $X \mapsto \alpha_X$ is continuous.
\end{proof}

\Cref{ass:method} enables us to develop the subsequent tools for general kernel-based approximation methods, 
not being limited to kernel interpolation.

\subsubsection{Definition of density functions}

In order to obtain a meaningful discretization of the integral of the kernel integral operator $T$, 
we will here consider sets of centers $X$ given by the intersection of a grid with the region $\Omega$.
To be precise, we define for $n \in \N$
\begin{align}
\begin{aligned}
\label{eq:def_grid_like}
Z_n :=&~ \{ z \in 2^{-n} \Z^d ~|~ z + [0, 2^{-n}]^d \subset \Omega \} , \\
\end{aligned}
\end{align}
Note that, by definition, we have $Z_n + b \in \Omega$ for all $b \in [0, 2^{-n}]^d$
and furthermore $Z_n \subseteq Z_{n+1} \subset \Omega$ for all $n \in \N$.
We make use of $Z$ instead of $X$ to distinguish these grid-based sets from possibly scattered sets.
If follows immediately that
\begin{align*}
Z_n + [0, 2^{-n}]^d \subset Z_{n+1} + [0, 2^{-(n+1)}]^d \subset \Omega.
\end{align*}
Such constructions are also considered in finite element literature, 
see e.g.\ Whitney decomposition.
If $\Omega$ is Jordan measurable, which is the case for Lipschitz regions,
then this approximation will fill $\Omega$ from inside, see \Cref{fig:vis_discretization_region}.

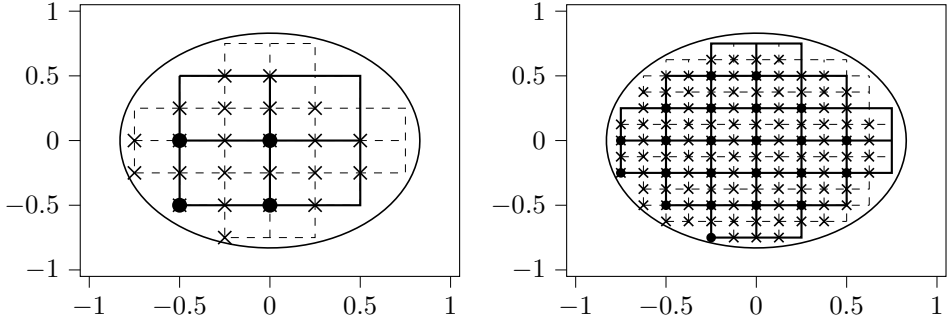
\begin{figure}[t]
\setlength\fwidth{.54\textwidth}
\centering
\begin{tikzpicture}

\definecolor{darkgray176}{RGB}{176,176,176}

\begin{axis}[
width=0.951\fwidth,
height=0.75\fwidth,
tick align=outside,
tick pos=left,
x grid style={darkgray176},
xmin=-1.05, xmax=1.05,
xtick style={color=black},
y grid style={darkgray176},
ymin=-1.05, ymax=1.05,
ytick style={color=black}
]
\addplot [thick, black]
table {%
-0.5 0.5
0.5 0.5
0.5 -0.5
-0.5 -0.5
-0.5 0.5
};
\addplot [thick, black]
table {%
-0.5 0
0.5 0
};
\addplot [thick, black]
table {%
0 0.5
0 -0.5
};
\addplot [line width=0.24pt, black, dashed]
table {%
-0.5 0.5
-0.5 0.25
-0.75 0.25
-0.75 -0.25
-0.5 -0.25
-0.5 -0.5
-0.25 -0.5
-0.25 -0.75
0.25 -0.75
0.25 -0.5
0.5 -0.5
0.5 -0.25
0.75 -0.25
0.75 0.25
0.5 0.25
0.5 0.5
0.25 0.5
0.25 0.75
-0.25 0.75
-0.25 0.5
-0.5 0.5
};
\addplot [line width=0.24pt, black, dashed]
table {%
-0.5 0.25
0.5 0.25
};
\addplot [line width=0.24pt, black, dashed]
table {%
-0.75 0
0.75 0
};
\addplot [line width=0.24pt, black, dashed]
table {%
-0.5 -0.25
0.5 -0.25
};
\addplot [line width=0.24pt, black, dashed]
table {%
0.25 -0.5
0.25 0.5
};
\addplot [line width=0.24pt, black, dashed]
table {%
0 -0.75
0 0.75
};
\addplot [line width=0.24pt, black, dashed]
table {%
-0.25 -0.5
-0.25 0.5
};
\addplot [semithick, black, mark=*, mark size=2.5, mark options={solid}, only marks]
table {%
-0.5 0
-0.5 -0.5
0 -0.5
0 0
};
\addplot [semithick, black, mark=x, mark size=3.5, mark options={solid}, only marks]
table {%
-0.25 0.5
0 0.5
-0.5 0.25
-0.25 0.25
0 0.25
0.25 0.25
-0.75 0
-0.5 0
-0.25 0
0 0
0.25 0
0.5 0
-0.75 -0.25
-0.5 -0.25
-0.25 -0.25
0 -0.25
0.25 -0.25
0.5 -0.25
-0.5 -0.5
-0.25 -0.5
0 -0.5
0.25 -0.5
-0.25 -0.75
};
\addplot [semithick, black]
table {%
0.83 0
0.828328941471664 0.0526418533149485
0.82332249464956 0.105071736466212
0.815000818728047 0.15707853281914
0.803397422158976 0.208452829360296
0.788559027724985 0.258987759929744
0.77054538440334 0.308479838198072
0.749429026777896 0.356729777034012
0.725294982967921 0.403543290963389
0.69824043224988 0.448731878488146
0.668374313750779 0.492113581115352
0.635816887788752 0.533513716039828
0.600699251627208 0.572765579530153
0.56316281159242 0.609711118185753
0.523358713680154 0.644201565362158
0.481447234944094 0.676098040201779
0.437597138116717 0.705272106858097
0.391984992061327 0.731606291661493
0.344794460791566 0.75499455614425
0.296215563921254 0.775342724020038
0.246443911522468 0.792568860398581
0.195679916472825 0.80660360170854
0.144127987463552 0.817390435000133
0.0919957059178392 0.824885925502841
0.039492990133706 0.829059891521897
-0.0131687499828907 0.829895525969316
-0.0657774641911345 0.827389464040112
-0.118121315766827 0.821551796761174
-0.169989534494108 0.812406030358266
-0.221173265362729 0.799988991604752
-0.27146640955346 0.784350679533175
-0.320666454325297 0.765554064106802
-0.368575288462793 0.743674832661809
-0.415 0.718801085141084
-0.459753653008871 0.69103297934686
-0.502656040324263 0.660482327640591
-0.543534409174587 0.627272146714034
-0.582224156796247 0.591536162244456
-0.618569493230876 0.553418270428522
-0.652424068636514 0.513071958563102
-0.683651562586761 0.4706596870061
-0.712126232985031 0.426352235005927
-0.737733422383586 0.380328013033751
-0.760370019668618 0.332772344387489
-0.779944875252304 0.283876718960305
-0.796379168100033 0.233838022178387
-0.809606723114888 0.182857742212829
-0.819574277601407 0.13114115865788
-0.82624169573566 0.0788965159424714
-0.829582130178044 0.0263341848033962
-0.829582130178044 -0.0263341848033963
-0.82624169573566 -0.0788965159424716
-0.819574277601407 -0.131141158657881
-0.809606723114888 -0.182857742212829
-0.796379168100033 -0.233838022178387
-0.779944875252304 -0.283876718960305
-0.760370019668618 -0.332772344387489
-0.737733422383586 -0.380328013033751
-0.712126232985031 -0.426352235005927
-0.683651562586761 -0.4706596870061
-0.652424068636513 -0.513071958563102
-0.618569493230876 -0.553418270428522
-0.582224156796247 -0.591536162244456
-0.543534409174587 -0.627272146714034
-0.502656040324263 -0.660482327640591
-0.459753653008872 -0.69103297934686
-0.415 -0.718801085141084
-0.368575288462792 -0.743674832661809
-0.320666454325297 -0.765554064106802
-0.27146640955346 -0.784350679533175
-0.221173265362729 -0.799988991604752
-0.169989534494108 -0.812406030358266
-0.118121315766827 -0.821551796761174
-0.0657774641911347 -0.827389464040112
-0.0131687499828903 -0.829895525969316
0.0394929901337062 -0.829059891521897
0.091995705917839 -0.824885925502841
0.144127987463552 -0.817390435000133
0.195679916472825 -0.80660360170854
0.246443911522468 -0.792568860398581
0.296215563921253 -0.775342724020039
0.344794460791566 -0.75499455614425
0.391984992061327 -0.731606291661493
0.437597138116717 -0.705272106858097
0.481447234944094 -0.676098040201779
0.523358713680154 -0.644201565362158
0.56316281159242 -0.609711118185753
0.600699251627208 -0.572765579530153
0.635816887788752 -0.533513716039828
0.668374313750779 -0.492113581115351
0.69824043224988 -0.448731878488146
0.725294982967921 -0.403543290963389
0.749429026777895 -0.356729777034013
0.77054538440334 -0.308479838198072
0.788559027724985 -0.258987759929744
0.803397422158976 -0.208452829360296
0.815000818728047 -0.15707853281914
0.82332249464956 -0.105071736466212
0.828328941471664 -0.0526418533149485
0.83 -2.03291368658461e-16
};
\end{axis}

\end{tikzpicture}
\begin{tikzpicture}

\definecolor{darkgray176}{RGB}{176,176,176}

\begin{axis}[
width=0.951\fwidth,
height=0.75\fwidth,
tick align=outside,
tick pos=left,
x grid style={darkgray176},
xmin=-1.05, xmax=1.05,
xtick style={color=black},
y grid style={darkgray176},
ymin=-1.05, ymax=1.05,
ytick style={color=black}
]
\addplot [thick, black]
table {%
-0.5 0.5
-0.5 0.25
-0.75 0.25
-0.75 -0.25
-0.5 -0.25
-0.5 -0.5
-0.25 -0.5
-0.25 -0.75
0.25 -0.75
0.25 -0.5
0.5 -0.5
0.5 -0.25
0.75 -0.25
0.75 0.25
0.5 0.25
0.5 0.5
0.25 0.5
0.25 0.75
-0.25 0.75
-0.25 0.5
-0.5 0.5
};
\addplot [thick, black]
table {%
-0.25 0.5
0.25 0.5
};
\addplot [thick, black]
table {%
-0.5 0.25
0.5 0.25
};
\addplot [thick, black]
table {%
-0.75 0
0.75 0
};
\addplot [thick, black]
table {%
-0.5 -0.25
0.5 -0.25
};
\addplot [thick, black]
table {%
-0.25 -0.5
0.25 -0.5
};
\addplot [thick, black]
table {%
-0.5 -0.25
-0.5 0.25
};
\addplot [thick, black]
table {%
-0.25 -0.5
-0.25 0.5
};
\addplot [thick, black]
table {%
0 -0.75
0 0.75
};
\addplot [thick, black]
table {%
0.25 -0.5
0.25 0.5
};
\addplot [thick, black]
table {%
0.5 -0.25
0.5 0.25
};
\addplot [line width=0.24pt, black, dashed]
table {%
-0.5 0.625
0.5 0.625
};
\addplot [line width=0.24pt, black, dashed]
table {%
-0.625 0.5
0.625 0.5
};
\addplot [line width=0.24pt, black, dashed]
table {%
-0.625 0.375
0.625 0.375
};
\addplot [line width=0.24pt, black, dashed]
table {%
-0.75 0.125
0.75 0.125
};
\addplot [line width=0.24pt, black, dashed]
table {%
-0.75 -0.125
0.75 -0.125
};
\addplot [line width=0.24pt, black, dashed]
table {%
-0.625 -0.375
0.625 -0.375
};
\addplot [line width=0.24pt, black, dashed]
table {%
-0.625 -0.5
0.625 -0.5
};
\addplot [line width=0.24pt, black, dashed]
table {%
-0.5 -0.625
0.5 -0.625
};
\addplot [line width=0.24pt, black, dashed]
table {%
-0.625 -0.5
-0.625 0.5
};
\addplot [line width=0.24pt, black, dashed]
table {%
-0.5 -0.625
-0.5 0.625
};
\addplot [line width=0.24pt, black, dashed]
table {%
-0.375 -0.625
-0.375 0.625
};
\addplot [line width=0.24pt, black, dashed]
table {%
-0.125 -0.75
-0.125 0.75
};
\addplot [line width=0.24pt, black, dashed]
table {%
0.125 -0.75
0.125 0.75
};
\addplot [line width=0.24pt, black, dashed]
table {%
0.375 -0.625
0.375 0.625
};
\addplot [line width=0.24pt, black, dashed]
table {%
0.5 -0.625
0.5 0.625
};
\addplot [line width=0.24pt, black, dashed]
table {%
0.625 -0.5
0.625 0.5
};
\addplot [semithick, black, mark=*, mark size=1.5, mark options={solid}, only marks]
table {%
-0.25 0.5
0 0.5
-0.5 0.25
-0.25 0.25
0 0.25
0.25 0.25
-0.75 0
-0.5 0
-0.25 0
0 0
0.25 0
0.5 0
-0.75 -0.25
-0.5 -0.25
-0.25 -0.25
0 -0.25
0.25 -0.25
0.5 -0.25
-0.5 -0.5
-0.25 -0.5
0 -0.5
0.25 -0.5
-0.25 -0.75
};
\addplot [semithick, black, mark=x, mark size=2.5, mark options={solid}, only marks]
table {%
0.125 0.625
0 0.625
-0.125 0.625
-0.25 0.625
0.375 0.5
0.25 0.5
0.125 0.5
0 0.5
-0.125 0.5
-0.25 0.5
-0.375 0.5
-0.5 0.5
0.5 0.375
0.375 0.375
0.25 0.375
0.125 0.375
0 0.375
-0.125 0.375
-0.25 0.375
-0.375 0.375
-0.5 0.375
-0.625 0.375
0.5 0.25
0.375 0.25
0.25 0.25
0.125 0.25
0 0.25
-0.125 0.25
-0.25 0.25
-0.375 0.25
-0.5 0.25
-0.625 0.25
0.625 0.125
0.5 0.125
0.375 0.125
0.25 0.125
0.125 0.125
0 0.125
-0.125 0.125
-0.25 0.125
-0.375 0.125
-0.5 0.125
-0.625 0.125
-0.75 0.125
0.625 0
0.5 0
0.375 0
0.25 0
0.125 0
0 0
-0.125 0
-0.25 0
-0.375 0
-0.5 0
-0.625 0
-0.75 0
0.625 -0.125
0.5 -0.125
0.375 -0.125
0.25 -0.125
0.125 -0.125
0 -0.125
-0.125 -0.125
-0.25 -0.125
-0.375 -0.125
-0.5 -0.125
-0.625 -0.125
-0.75 -0.125
0.625 -0.25
0.5 -0.25
0.375 -0.25
0.25 -0.25
0.125 -0.25
0 -0.25
-0.125 -0.25
-0.25 -0.25
-0.375 -0.25
-0.5 -0.25
-0.625 -0.25
-0.75 -0.25
0.5 -0.375
0.375 -0.375
0.25 -0.375
0.125 -0.375
0 -0.375
-0.125 -0.375
-0.25 -0.375
-0.375 -0.375
-0.5 -0.375
-0.625 -0.375
0.5 -0.5
0.375 -0.5
0.25 -0.5
0.125 -0.5
0 -0.5
-0.125 -0.5
-0.25 -0.5
-0.375 -0.5
-0.5 -0.5
-0.625 -0.5
0.375 -0.625
0.25 -0.625
0.125 -0.625
0 -0.625
-0.125 -0.625
-0.25 -0.625
-0.375 -0.625
-0.5 -0.625
0.125 -0.75
0 -0.75
-0.125 -0.75
};
\addplot [semithick, black]
table {%
0.83 0
0.828328941471664 0.0526418533149485
0.82332249464956 0.105071736466212
0.815000818728047 0.15707853281914
0.803397422158976 0.208452829360296
0.788559027724985 0.258987759929744
0.77054538440334 0.308479838198072
0.749429026777896 0.356729777034012
0.725294982967921 0.403543290963389
0.69824043224988 0.448731878488146
0.668374313750779 0.492113581115352
0.635816887788752 0.533513716039828
0.600699251627208 0.572765579530153
0.56316281159242 0.609711118185753
0.523358713680154 0.644201565362158
0.481447234944094 0.676098040201779
0.437597138116717 0.705272106858097
0.391984992061327 0.731606291661493
0.344794460791566 0.75499455614425
0.296215563921254 0.775342724020038
0.246443911522468 0.792568860398581
0.195679916472825 0.80660360170854
0.144127987463552 0.817390435000133
0.0919957059178392 0.824885925502841
0.039492990133706 0.829059891521897
-0.0131687499828907 0.829895525969316
-0.0657774641911345 0.827389464040112
-0.118121315766827 0.821551796761174
-0.169989534494108 0.812406030358266
-0.221173265362729 0.799988991604752
-0.27146640955346 0.784350679533175
-0.320666454325297 0.765554064106802
-0.368575288462793 0.743674832661809
-0.415 0.718801085141084
-0.459753653008871 0.69103297934686
-0.502656040324263 0.660482327640591
-0.543534409174587 0.627272146714034
-0.582224156796247 0.591536162244456
-0.618569493230876 0.553418270428522
-0.652424068636514 0.513071958563102
-0.683651562586761 0.4706596870061
-0.712126232985031 0.426352235005927
-0.737733422383586 0.380328013033751
-0.760370019668618 0.332772344387489
-0.779944875252304 0.283876718960305
-0.796379168100033 0.233838022178387
-0.809606723114888 0.182857742212829
-0.819574277601407 0.13114115865788
-0.82624169573566 0.0788965159424714
-0.829582130178044 0.0263341848033962
-0.829582130178044 -0.0263341848033963
-0.82624169573566 -0.0788965159424716
-0.819574277601407 -0.131141158657881
-0.809606723114888 -0.182857742212829
-0.796379168100033 -0.233838022178387
-0.779944875252304 -0.283876718960305
-0.760370019668618 -0.332772344387489
-0.737733422383586 -0.380328013033751
-0.712126232985031 -0.426352235005927
-0.683651562586761 -0.4706596870061
-0.652424068636513 -0.513071958563102
-0.618569493230876 -0.553418270428522
-0.582224156796247 -0.591536162244456
-0.543534409174587 -0.627272146714034
-0.502656040324263 -0.660482327640591
-0.459753653008872 -0.69103297934686
-0.415 -0.718801085141084
-0.368575288462792 -0.743674832661809
-0.320666454325297 -0.765554064106802
-0.27146640955346 -0.784350679533175
-0.221173265362729 -0.799988991604752
-0.169989534494108 -0.812406030358266
-0.118121315766827 -0.821551796761174
-0.0657774641911347 -0.827389464040112
-0.0131687499828903 -0.829895525969316
0.0394929901337062 -0.829059891521897
0.091995705917839 -0.824885925502841
0.144127987463552 -0.817390435000133
0.195679916472825 -0.80660360170854
0.246443911522468 -0.792568860398581
0.296215563921253 -0.775342724020039
0.344794460791566 -0.75499455614425
0.391984992061327 -0.731606291661493
0.437597138116717 -0.705272106858097
0.481447234944094 -0.676098040201779
0.523358713680154 -0.644201565362158
0.56316281159242 -0.609711118185753
0.600699251627208 -0.572765579530153
0.635816887788752 -0.533513716039828
0.668374313750779 -0.492113581115351
0.69824043224988 -0.448731878488146
0.725294982967921 -0.403543290963389
0.749429026777895 -0.356729777034013
0.77054538440334 -0.308479838198072
0.788559027724985 -0.258987759929744
0.803397422158976 -0.208452829360296
0.815000818728047 -0.15707853281914
0.82332249464956 -0.105071736466212
0.828328941471664 -0.0526418533149485
0.83 -2.03291368658461e-16
};
\end{axis}

\end{tikzpicture}
\caption{Visualization of the approximation of $\Omega$ from the interior with help of $Z_n + [0, 2^{-n}]^d$.
In this example, $\Omega$ is given by a circle of radius $0.83$, and $Z_n$ (solid dots) and $Z_{n+1}$ (crosses) are visualized for $n=1$ (left) and $n=2$ (right).}
\label{fig:vis_discretization_region}
\end{figure}

As a first technical statement, we characterize the fill- and seperation distance of $Z_n + b$ for any $b \in [0, 2^{-n}]^d$.
Essentially we show that the sequence of sets $(Z_n)_{n \in \N}$ is quasi-uniform, 
i.e.\ it satisfies \Cref{ass:points}.
The statement itself is probably known in similar forms in the literature, however we could not locate it. 
Thus we provide the easy geometrical proof, 
including a visualization of the proof idea in \Cref{fig:vis_cone_construction}.

\begin{figure}[t]
\setlength\fwidth{.4\textwidth}
\centering
\input{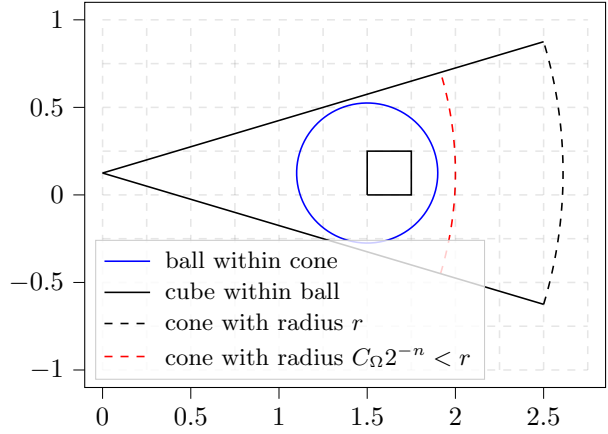}
\caption{Visualization of the proof of \Cref{prop:estimate_sep_fill_dist_Z}: 
For any $x \in \Omega$, the cone condition provides a ball and thus a cube within the region,
which is not too far from the point $x$.}
\label{fig:vis_cone_construction}
\end{figure}

\begin{prop}
\label{prop:estimate_sep_fill_dist_Z}
Let $\Omega \subset \R^d$ be a Lipschitz region.
Then there exists $n_0 \in \N$ and constants $c_\Omega, C_\Omega > 0$ such that for all $n \geq n_0$ the point sets $Z_{n}+b$ for any $b \in [0, 2^{-n}]^d$ satisfy the following uniformity bounds:
\begin{align*}
c_\Omega 2^{-n} = c_\Omega q_{Z_n + b} \leq h_{Z_n + b} \leq C_\Omega q_{Z_n + b} = C_\Omega 2^{-n}.
\end{align*}
\end{prop}

\begin{proof}
Since $\Omega$ is a Lipschitz region, it satisfies an interior cone condition \cite[Lemma 1.5]{krieg2024random}.
Let $r$ and $\alpha$ be the radius and angle of this cone condition.
We define $C_\Omega := 2\sqrt{d} \cdot \frac{1+\sin(\alpha)}{\sin(\alpha)} > 1$ and $n_0 \in \N$ as the smallest integer such that $C_\Omega 2^{-n_0} < r$.

We first show an upper bound on the fill distance,
which is defined as $h_{Z_n} \equiv \sup_{x \in \Omega} \min_{z \in Z_n}  \Vert x - z \Vert_2$.
Thus let $x \in \Omega$.
Due to the interior cone condition,
we have a cone $C(x, \xi(x), \alpha, r)$ within $\Omega$, i.e.
\begin{align*}
C(x, \xi(x), \alpha, r) \subset \Omega.
\end{align*}
For $n \geq n_0$ we consider the smaller cone $C(x, \xi(x), \alpha, C_\Omega 2^{-n})$.
By definition of $n_0$ we have for $n \geq n_0$ that $C_\Omega 2^{-n} \leq C_\Omega 2^{-n_0} < r$, i.e.\
\begin{align*}
C(x, \xi(x), \alpha, C_\Omega 2^{-n}) \subset C(x, \xi(x), \alpha, r).
\end{align*}
By \cite[Lemma 1.1]{krieg2024random}, there exists a ball within this cone with radius $\frac{\sin(\alpha)}{1+\sin(\alpha)} C_\Omega 2^{-n} = 2\sqrt{d} 2^{-n}$ centered at some $\tilde{x} \in C(x, \xi(x), \alpha, C_\Omega 2^{-n})$,
i.e.\
\begin{align*}
B_{2\sqrt{d}2^{-n}}(\tilde{x}) \subset C(x, \xi(x), \alpha, C_\Omega 2^{-n}).
\end{align*}
We now show that within this ball $B_{2\sqrt{d}2^{-n}}(\tilde{x})$, there exists a cube $z + [0, 2^{-n}]^d$ for some $z \in \Z_n$:
The distance of $\tilde{x}$ to $2^{-n}\Z^d$ is at most $\frac{1}{2}\sqrt{d} \cdot 2^{-n}$,
while the diameter of a cube of side length $2^{-n}$ is $\sqrt{d} 2^{-n}$.
Since the ball has a radius of $2\sqrt{d}2^{-n} > \frac{1}{2} \sqrt{d} 2^{-n} + \sqrt{d} 2^{-n}$,
there exists a $z \in 2^{-n}\Z^d$ such that $z + [0, 2^{-n}]^d$ is included in that ball, 
i.e.\
\begin{align*}
z + [0, 2^{-n}]^d \subset B_{2\sqrt{d}2^{-n}}(\tilde{x}) \subset \Omega \qquad \text{for } z \in 2^{-n}\Z^d.
\end{align*}
Since this cube is fully contained within $\Omega$, it follows $z \in Z_n$.
Now we may finally estimate the distance of $x$ to $Z_n + b$ for any $b \in [0, 2^{-n}]^d$:
Due to $x \in C(x, \xi(x), \alpha, C_\Omega 2^{-n})$ and $z + [0, 2^{-n}]^d \subset C(x, \xi(x), \alpha, C_\Omega 2^{-n})$ we have $\mathrm{dist}(x, Z_n + b) \leq C_\Omega 2^{-n}$
and thus $h_{Z_n + b} \leq C_\Omega 2^{-n}$.

Now we show an estimate on the seperation distance:
Since $Z_n \subset 2^{-n} \Z^d$ for any $n \in \N$, we have immediately $q_{Z_n} = 2^{-n}$ and thus $q_{Z_n + b} = 2^{-n}$ for any $n \in \N$.

Combining both estimates, we obtain the uniformity bound $\frac{h_{Z_n + b}}{q_{Z_n + b}} \leq C_\Omega \Leftrightarrow h_{Z_n + b} \leq C_\Omega q_{Z_n + b}$
for any $b \in [0, 2^{-n}]^d$.

The remaining bound $c_\Omega q_{Z_n+b} \leq h_{Z_n+b}$ is not specific to the set $Z_n+b$,
in fact it holds $c_\Omega q_X \leq h_X$ for any set of points $X \subset \Omega$ by a volume comparison arguments, see Eq.~\eqref{eq:bound_fill_sep_dist}.
\end{proof}

Next, given a continuous function $f \in \mathcal{C}(\Omega)$ we define a piecewise continuous density function $D_{Z_n}(f): \Omega \rightarrow \R$ based on the coefficients $(\alpha_j)_j$ of the kernel-based approximant:
Since $Z_n$ is a subset of $2^{-n} \Z^d$,
there is a unique decomposition of $x \in Z_n + [0, 2^{-n})^d \subset \Omega$ as
\begin{align}
\label{eq:unique_representation_x}
x = z(x) + b(x), \qquad z(x) \in Z_n, ~ b(x) \in [0, 2^{-n})^d.
\end{align}

\begin{definition}
Consider $f \in \mathcal{C}(\Omega)$ 
and a given kernel-based approximation scheme in Eq.~\eqref{eq:def_coefficients_D_new}, e.g.\ kernel-based interpolation.
\begin{itemize}
\item For $x \in \Omega \setminus (Z_n + [0, 2^{-n}))^d$,
we set $D_{Z_n}(f)(x) := 0$.
\item For $x \in Z_n + [0, 2^{-n})^d \subset \Omega$, 
we consider the unique decomposition of Eq.~\eqref{eq:unique_representation_x} as $x = z(x) + b(x)$ and consider the kernel approximant $s_{f, Z_n+b(x)}$ to $f$ based on the centers $Z_n + b(x)$,
\begin{align}
\label{eq:def_coefficients_D_new}
s_{f, Z_n + b(x)} &=: \sum_{x_i \in Z_n  + b} \alpha_{f, Z_n + b(x); x_i} k(\cdot, x_i).
\end{align}
Then we define $D_{Z_n}(f)$ via
\begin{align}
\label{eq:def_operator_D_new}
D_{Z_n}(f)(x) := 2^{nd} \cdot \alpha_{f, Z_n+b(x); x}
\end{align}
\end{itemize}
\end{definition}
The definition is well-defined, because for specified $n \in \N$, 
there is a unique $b(x) \in [0, 2^{-n})^d$ and a unique $z \in 2^{-n}\Z^d$ such that
$x = b(x) + z$, 
see Eq.~\eqref{eq:unique_representation_x}.
Note that this definition does not specify, how the kernel approximants $s_{f, X}$ actually looks like:
it does not rely on being a kernel-based interpolant,
and thus also works for e.g.\ regularized kernel-based approximation or kernel Galerkin methods.
A visualization of the functions $D_{Z_n}(f)$ obtained by kernel interpolation for $n \in \{ 3, 4, 5 \}$ is provided in \Cref{fig:vis_density_func_approx}.

\begin{figure}[t]
\centering
\input{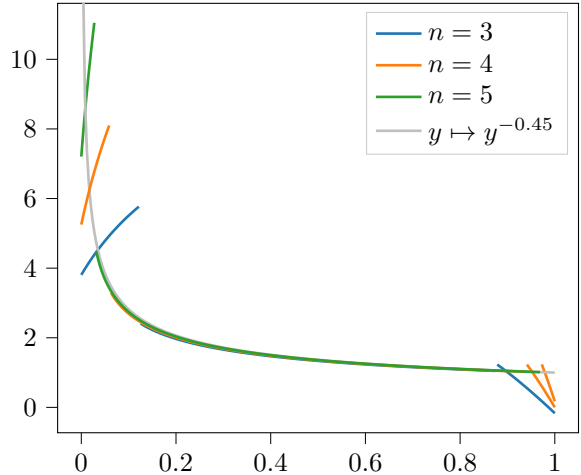}
\caption{Visualization of the density functions $D_{Z_n}(f)$ obtained by kernel interpolation for $n \in \{3, 4, 5\}$ for $\Omega = [0, 1]$ and $f$ given by $f = Tv$ with $v(z) = z^{-0.45} \in L_2(\Omega)$.
The piecewise continuous nature of these functions is clearly observeable.
The Wendland kernel $k(x, z) = \max(1-|x-z|, 0)$ was used for this example.}
\label{fig:vis_density_func_approx}
\end{figure}

\subsubsection{Properties of density functions}

Next, we will derive several necessary properties of the constructed density functions $D_{Z_n}(f)$ and of $TD_{Z_n}(f)$.
First we need to show that the constructed density function $D_{Z_n}(f)$ is piecewiese continuous.
This is the case, if the coefficients $\alpha_{f, Z_n + b(x); x_i}$ of Eq.~\eqref{eq:def_coefficients_D_new} depend in a continuous way on $b(x)$,
i.e.\ under \Cref{ass:method}.

\begin{prop}
\label{prop:piecewise_continuous}
Let $k: \Omega \times \Omega \rightarrow \R$ be a continuous strictly positive definite kernel defined on a compact Lipschitz region $\Omega \subset \R^d$.
Consider a kernel-based approximation method satisfying \Cref{ass:method}.

Given $f \in \mathcal{C}(\Omega)$, 
it follows that the density function $D_{Z_n}(f)$ is continuous on $z + (0, 2^{-n})^d$ for any $z \in 2^{-n} \Z^d$.
Furthermore, $D_{Z_n}(f)|_{z + (0, 2^{-n})^d}$ can be extended continuously to $z + [0, 2^{-n}]^d$.
\end{prop}

\begin{proof}
For $z \in 2^{-n} \Z^d \setminus Z_n$,
the density function $D_{Z_n}(f)$ is defined to be equal zero and thus continuous.
For $z \in Z_n \subset \Omega$, the function values are obtained by Eq.~\eqref{eq:def_operator_D_new}.
Since the kernel-based approximation method is assumed to satisfy \Cref{ass:method},
the coefficients $(\alpha_{f, Z_n+b(x); x_i})_{i=1, ..., |Z_n+b(x)|}$ depend continuously on $Z_n+b(x)$ and thus on $b(x)$,
where $b(x) \in (0, 2^{-n})^d$.
Recall that it holds $z \in Z_n$ if and only if $z + [0, 2^{-n}]^d \subset \Omega$.
Thus the coefficient vector $(\alpha_{f, Z_n+b(x); x_i})_{i=1, ..., |Z_n+b(x)|}$ (as a function in $b(x)$) can be extended continuously to $z + [0, 2^{-n}]^d$.
\end{proof}

From the previous proposition, we can immediately conclude that $D_{Z_n}(f)$ is bounded and thus $D_{Z_n}(f) \in L_2(\Omega)$.
Therefore $TD_{Z_n}(f) \in TL_2(\Omega) \subset \mathcal{C}(\Omega)$.

Next, we would like to show that $TD_{Z_n}(f)$ converges to $f$ in $L_2(\Omega)$ for $n \rightarrow \infty$.
This can even be quantified:
\begin{prop}
\label{prop:approx_by_TD}
Given $f \in \mathcal{C}(\Omega)$,
assume it holds 
\begin{align*}
\Vert f - s_{f, X} \Vert_{L_2(\Omega)} \leq Ch_X^\beta
\end{align*}
for some $\beta > 0$ and all quasi-uniform $X \subset \Omega$.
Then it holds
\begin{align*}
\Vert f - TD_{Z_n}(f) \Vert_{L_2(\Omega)} \leq C 2^{-n\beta}.
\end{align*}
\end{prop}
Note that the rate of convergence $2^{-n\beta}$ matches the rate given by $h_X^\beta$,
but expressed in terms of the discretization level $n$ of $Z_n$.
Essentially, the assumed approximation rate of $s_{f, X}$ carries over to $TD_{Z_n}(f)$ without any loss in the rate.

\begin{proof}
In order to show that the convergence rate on $\Vert f - s_{f, X} \Vert_{L_2(\Omega)}$ basically carries over to $\Vert f - TD_{Z_n}(f) \Vert_{L_2(\Omega)}$,
we make use of Eq.~\eqref{eq:motivation_dirac} and the piecewise continuity of $TD_{Z_n}(f)$,
which allows to discretize the $L_2(\Omega)$ integral with help of Riemann sums.
We have:
\begin{align*}
\Vert f - TD_{Z_n}(f) \Vert_{L_2(\Omega)}^2 
&= \int_\Omega \left| f(x) - TD_{Z_n}(f)(x) \right|^2 ~ \mathrm{d}x
\end{align*}
We compute by using the definition of $D_{Z_n}(f)$
\begin{align*}
TD_{Z_n}(f) &= \int_\Omega k(\cdot, y) D_{Z_n}(f)(z) ~ \mathrm{d}z \\
&= \int_{Z_n + [0, 2^{-n}]^d} k(\cdot, z) D_{Z_n}(f)(z) ~ \mathrm{d}z \\
&= \sum_{z \in Z_n} \int_{z + [0, 2^{-n}]^d} k(\cdot, z) D_{Z_n}(f)(z) ~ \mathrm{d}z \\
&= \sum_{z \in Z_n} \lim_{N \rightarrow \infty} \frac{2^{-nd}}{N^d} \cdot \sum_{b \in \frac{2^{-n}}{N} \{ 0, ..., N-1 \}^d } k(\cdot, z+b) 2^{nd} \alpha_{f, Z_n + b; z+b}.
\end{align*}
In the last step, the integral was replaced by the limit of an Riemann approximation using grid points $b$ in $2^{-n} [0, 1]^d$. 
This is possible because of the piecewise continuity, 
guaranteed by \Cref{prop:piecewise_continuous}.
Some simplifications and rearrangements of the sums yield
\begin{align*}
&= \lim_{N \rightarrow \infty} \frac{1}{N^d} \cdot \sum_{b \in \frac{2^{-n}}{N} \{ 0, ..., N-1 \}^d } \sum_{z \in Z_n} k(\cdot, z+b) \alpha_{f, Z_n + b; z+b}  \\
&= \lim_{N \rightarrow \infty} \frac{1}{N^d} \cdot \sum_{b \in \frac{2^{-n}}{N} \{ 0, ..., N-1 \}^d } s_{f, Z_n+b}.
\end{align*}
Thus we can finally estimate:
\begin{align*}
\Vert f - TD_{Z_n}(f) \Vert_{L_2(\Omega)}
&= \left \Vert f - \lim_{N \rightarrow \infty} \frac{1}{N^d} \cdot \sum_{b \in \frac{2^{-n}}{N} \{ 0, ..., N-1 \}^d } s_{f, Z_n+b} \right \Vert_{L_2(\Omega)} \\
&\leq \lim_{N \rightarrow \infty} \frac{1}{N^d} \sum_{b \in \frac{2^{-n}}{N} \{ 0, ..., N-1 \}^d } \left \Vert f -  s_{f, Z_n+b} \right \Vert_{L_2(\Omega)} \\
&\leq \lim_{N \rightarrow \infty} \frac{1}{N^d} \sum_{b \in \frac{2^{-n}}{N} \{ 0, ..., N-1 \}^d } C h_{Z_n + b}^{\beta}.
\end{align*}
The fill distances $h_{Z_n+b}$ for any $b \in [0, 2^{-n}]^d$ can be bounded by \Cref{prop:estimate_sep_fill_dist_Z} as $C_\Omega 2^{-n}$,
such that the final result follows immediately.
\end{proof}

Next,
as a preparation for showing a Cauchy sequence property in certain spaces (see proof of \Cref{th:main_result}),
we are interested in precisely bounding $\Vert D_{Z_n} f - D_{Z_{n+1}} f \Vert_{L_2(\Omega)}$.
For this we require the following utility statements on the norm of the coefficient vector $(\alpha_j)_{j=1}^{|X|}$.
The statement is  basic combination of \cite[Proposition 6]{wenzel2025sharp} and \Cref{prop:escaping_bernstein}:
\begin{prop}
\label{prop:estimate_lp_coeff_norm}
Let $\alpha = (\alpha_j)_{j=1}^{|X|} \in \R^{|X|}$ be the vector of the coefficients of $s_X = \sum_{j=1}^{|X|} \alpha_j k(\cdot, x_j)$ and $X:= \{x_1, ..., x_{|X|} \}$.
Then it holds
\begin{align}
\label{eq:bound_coeffs}
\Vert \alpha \Vert_{\ell_2} \leq Cq_X^{d/2-2\tau} \Vert s_X \Vert_{L_2(\Omega)}.
\end{align}
\end{prop}
\begin{proof}
We compute
\begin{align*}
\Vert \alpha \Vert_{\ell_2}^2 &= \alpha^\top A_{X}^{1/2} A_{X}^{-1} A_{X}^{1/2} \alpha \leq \Vert A_{X}^{-1} \Vert_{2,2} \cdot \alpha^\top A_{X} \alpha \notag \\
&\leq \Vert A_{X}^{-1} \Vert_{2,2} \cdot \Vert s_X \Vert_{\ns}^2 \\
&\leq Cq_X^{d-2\tau} q_X^{-2\tau} \Vert s_X \Vert_{L_2(\Omega)}^2,
\end{align*}
where we used both Eq.~\eqref{eq:estimate_lambda_min} and Eq.~\eqref{eq:escaping_bernstein} in the last step.
\end{proof}
Note that inequalities like Eq.~\eqref{eq:bound_coeffs} using $\ell_p$ and $L_p$ norms were discussed and derived in \cite[Theorem 5.3]{ward2012lp},
however using $L_p(\R^d)$ instead of $L_p(\Omega)$.
Thus \Cref{prop:estimate_lp_coeff_norm} provides a localized version of these for the case of $p=2$.

With this proposition at hand,
we can prove the following $L_2(\Omega)$ bound on $D_{Z_n}(f)$.
Note that for $\beta < 2\tau$, 
the upper bound grows in $n$.
\begin{prop}
\label{prop:L2_bound_density_approx_2}
Let $f \in \mathcal{C}(\Omega)$,
assume it holds 
\begin{align}
\label{eq:assumed_decay_rate}
\Vert f - s_{f, X} \Vert_{L_2(\Omega)} \leq C h_X^{\beta}
\end{align}
for some $\beta > 0$ and all quasi-uniform $X \subset \Omega$. 
Then it holds
\begin{align*}
\Vert D_{Z_n} f \Vert_{L_2(\Omega)} \leq C 2^{-n(\beta-2\tau)}.
\end{align*}
\end{prop}

\begin{proof}
First we would like to estimate the norm of the coefficient vector $\alpha_{f, Z_n + b} \in \R^{|Z_n|}$ of
$s_{f, Z_n + b} \equiv \sum_{x_i \in Z_n+b} \alpha_{f, Z_n + b; x_i} k(\cdot, x_i)$ for any $b \in [0, 2^{-n}]^d$,
see Eq.~\eqref{eq:def_coefficients_D_new}:
For $b \in [0, 2^{-n}]^d$ we use $b' \in \{ -2^{-(n+1)}, 2^{-(n+1)} \}^{d} \in \R^d$ such that
$Z_n + b + b' \in Z_n + [0, 2^{-n}]^d \subset \Omega$.
Then we have
\begin{align*}
\Vert s_{f, Z_n + b} - s_{f, Z_n + b + b'} \Vert_{L_2(\Omega)} 
&\leq \Vert s_{f, Z_n + b} - f \Vert_{L_2(\Omega)} + \Vert f - s_{f, Z_n + b + b'} \Vert_{L_2(\Omega)} \\
&\leq  C h_{Z_n + b}^\beta + C h_{Z_n + b + b'}^\beta \\
&\leq C 2^{-n\beta},
\end{align*}
where we used in the final step \Cref{prop:estimate_sep_fill_dist_Z} to estimate the fill distances.
Since $Z_n+b$ and $Z_n+b+b'$ are disjoint,
the squared norm of the vector of coefficients of $s_{f, Z_n + b} - s_{f, Z_n + b + b'}$
is given by the sum of squared norms of the vector of coefficients of $s_{f, Z_n + b}$ and of $s_{f, Z_n + b + b'}$,
i.e.\ by $\Vert \alpha_{f, Z_n + b} \Vert_{\ell_2}^2 + \Vert \alpha_{f, Z_n + b + b'} \Vert_{\ell_2}^2$.
Thus by \Cref{prop:estimate_lp_coeff_norm} we obtain
\begin{align*}
\Vert \alpha_{f, Z_n + b} \Vert_{\ell_2}^2 + \Vert \alpha_{f, Z_n + b + b'} \Vert_{\ell_2}^2 
&\leq C^2 q_{(Z_n + b) \cup (Z_n + b + b')}^{d-4\tau} \Vert s_{f, Z_n + b} - s_{f, Z_n + b + b'} \Vert_{L_2(\Omega)}^2 \\
&\leq C 2^{-(n+1)(d-4\tau)} 2^{-2n\beta} \\
&\leq C 2^{-nd - 2n(\beta - 2\tau)},
\end{align*}
and in particular
\begin{align}
\label{eq:intermediate_eq_1}
\Vert \alpha_{f, Z_n + b} \Vert_{\ell_2}^2 &\leq C 2^{-nd - 2n(\beta - 2\tau)}.
\end{align}
With this bound, which holds for any $b \in [0, 2^{-n}]^d$,
we can bound $\Vert D_{Z_n}(f) \Vert_{L_2(\Omega)}$:
For this, we start by bounding for any $b \in [0, 2^{-n}]^d$:
\begin{align*}
\Vert D_{Z_n}(f) \Vert_{\ell_2(Z_n + b)}^2
:=&~ \sum_{z \in Z_n + b} |D_{Z_n}(f)(z)|^2 \\
=&~ \sum_{z \in Z_n + b} |2^{nd} \alpha_{f, Z_n + b; z}|^2 \\
=&~ 2^{2nd} \cdot \Vert \alpha_{f, Z_n+b} \Vert_{\ell_2}^2 \\
\leq&~ C 2^{nd-2n(\beta-2\tau)}.
\end{align*}
Due to the piecewise continuity according to \Cref{prop:piecewise_continuous},
we can approximate the integral $\Vert D_{Z_n}(f) \Vert_{L_2(\Omega)}$ via Riemann sums:
\begin{align*}
\Vert D_{Z_n}(f) \Vert_{L_2(\Omega)}^2 
&= \int_\Omega |D_{Z_n}(f)(x)|^2 ~ \mathrm{d}x \\
&= \sum_{z \in Z_n} \int_{z + [0, 2^{-n}]^d} |D_{Z_n}(f)(x)|^2 ~ \mathrm{d}x \\
&= \sum_{z \in Z_n} \lim_{N \rightarrow \infty} \frac{2^{-nd}}{N^d} \sum_{b \in \frac{2^{-n}}{N} \{ 0, ..., N-1\}^d} |D_{Z_n}(f)(z + b)|^2 \\
&= \lim_{N \rightarrow \infty} \frac{2^{-nd}}{N^d} \sum_{b \in \frac{2^{-n}}{N} \{ 0, ..., N-1\}^d} \sum_{z \in Z_n} |D_{Z_n}(f)(z + b)|^2 \\
&\leq \lim_{N \rightarrow \infty} \frac{2^{-nd}}{N^d}
\sum_{b \in \frac{2^{-n}}{N} \{ 0, ..., N-1\}^d} C 2^{nd - 2n(\beta-2\tau)} \\
&= C 2^{-2n(\beta-2\tau)}.
\end{align*}
\end{proof}
\noindent Note that 
\Cref{prop:L2_bound_density_approx_2} immediately implies the bound
\begin{align*}
\Vert D_{Z_n} f - D_{Z_{n+1}} f \Vert_{L_2(\Omega)} \leq C 2^{n(2\tau-\beta)}
\end{align*}
by triangle inequality.

\section{Sharp inverse statements and saturation} \label{sec:cont_superconv_inverse}

This section states and proves the main result in \Cref{subsec:main_result},
and discusses it subsequently in \Cref{subsec:discussion_main_result}.

\subsection{Main result}
\label{subsec:main_result}

We restate the main theorem, 
briefly elaborate on the proof strategy and then have the full proof:

\mainresult*

Before providing the proof, we briefly elaborate on the proof strategy:
The case $\beta \in (0, \tau]$ was covered in \cite{wenzel2025sharp} and refined in \cite{avesani2025sobolev},
and is here addressed again due to a mild change in the assumptions.
For its proof, one considers the sequence of approximants $(s_{f, X_n})_{n \in \N}$ for a sequence $(X_n)_{n \in \N}$ satisfying \Cref{ass:points}.
Using the Bernstein inequality from \Cref{prop:escaping_bernstein},
one can show that this sequence is a Cauchy sequence in $\calh_{\vartheta'}(\Omega)$ for any $\vartheta' < \frac{\beta}{\tau}$.
Standard arguments then conclude the proof.
For the new case of $\beta \in (\tau, 2\tau - \frac{d}{2})$,
one can consider the same sequence $(s_{f, X_n})_{n \in \N}$, as actually $k(\cdot, x) \in \calh_\vartheta(\Omega)$ for $\vartheta < 2 - \frac{d}{2\tau}$ by Eq.~\eqref{eq:kernel_inclusion}.  
Using the novel superconvergence Bernstein inequality from \Cref{prop:superconv_bernstein},
the same proof strategy still works.
The other new case of $\beta \in (2\tau - \frac{d}{2}, 2\tau]$ is the more challenging one, as it holds $k(\cdot, x) \notin \calh_\vartheta(\Omega)$,
such that the previous proof strategy does no longer work.
Therefore the idea is to leverage Eq.~\eqref{eq:power_space_via_closure},
and use approximants from $TL_2(\Omega)$ instead.
The construction of such density functions $(D_{Z_n}(f))_{n \in \N}$ was derived and discussed in \Cref{subsec:construction_density_func}, 
based on the idea that $k(\cdot, x)$ can be approximated via $T\delta_{x, \varepsilon} = \int_\Omega k(\cdot, z)\delta_{x, \varepsilon}(z) ~ \mathrm{d}z$,
where $\delta_{x, \varepsilon}$ is an approximation of the Dirac delta for $\varepsilon \rightarrow 0$.
For the final case of $\beta > 2\tau$,
one considers again the sequence of density functions $(D_{Z_n}(f))_{n \in \N}$ and shows that $D_{Z_n}(f) \stackrel{n \rightarrow \infty}{\longrightarrow} 0$,
which then implies $f=0$.

\begin{proof}[Proof of \Cref{th:main_result}]
We distinguish three cases, depending on the range of $\vartheta$: \\
\textbf{First case:} For $\beta \in (0, 2\tau-\frac{d}{2}]$, we use the same constructive idea as in \cite{schaback2002inverse,wenzel2025sharp,avesani2025sobolev},
because we have $k(\cdot, x) \in \Ht$ for $\vartheta \in \left( 0, 2-\frac{d}{2\tau} \right)$ by Eq.~\eqref{eq:kernel_inclusion}.
Thus we apply Eq.~\eqref{eq:superconv_bernstein} to $s_{f, X_{n+1}} - s_{f, X_{n}}$:
\begin{align*}
\Vert s_{f, X_{n+1}} - s_{f, X_{n}} \Vert_{\calh_{\vartheta}(\Omega)} 
\leq C q_{X_{n+1}}^{-\vartheta\tau} \Vert s_{f, X_{n+1}} - s_{f, X_{n}} \Vert_{L_2(\Omega)}.
\end{align*}
Inserting $-f + f$, we can estimate $\Vert s_{f, X_{n+1}} - s_{f, X_{n}} \Vert_{L_2(\Omega)}$ by the assumption Eq.~\eqref{eq:inverse_statement_assumption}, 
and subsequently estimate the fill distances via the assumption on the sequence of points $(X_n)_{n \in \N}$ as formulated in \Cref{ass:points}:
\begin{align*}
\Vert s_{f, X_{n+1}} - s_{f, X_{n}} \Vert_{\calh_{\vartheta}(\Omega)} 
&\leq C q_{X_{n+1}}^{-\vartheta\tau} \left( c_f h_{X_{n+1}}^\beta + c_f h_{X_n}^\beta \right) \\
&\leq C a^{-(n+1) \vartheta \tau} a^{n\beta} \\
&\leq C a^{n(\beta-\vartheta \tau)}.
\end{align*}
For $\beta - \vartheta \tau > 0$, the exponent is positive, such that we are able to show that $(s_{f, X_n})_{n \in \N}$ is actually a Cauchy sequence:
Let $n > m \geq m_0$:
\begin{align*}
\Vert s_{f, X_{n}} - s_{f, X_m} \Vert_{\mathcal{H}_{\vartheta'}(\Omega)} &= \left \Vert \sum_{\ell=m}^{n-1} s_{f, X_{\ell+1}} - s_{f, X_{\ell}} \right \Vert_{\mathcal{H}_{\vartheta'}(\Omega)} \notag \\
&\leq \sum_{\ell=m}^{n-1} \Vert s_{f, X_{\ell+1}} - s_{f, X_{\ell}} \Vert_{\mathcal{H}_{\vartheta'}(\Omega)} \notag \\
&\leq \sum_{\ell=m}^{\infty} a^{\ell (\beta - \vartheta \tau)}
\leq \frac{a^{m_0 (\beta - \vartheta \tau)}}{1-a^{\beta - \vartheta \tau}} \stackrel{m_0 \rightarrow \infty}{\longrightarrow} 0
\end{align*}
due to $a \in (0, 1)$ and $\beta - \vartheta \tau > 0 \Leftrightarrow \vartheta < \frac{\beta}{\tau}$.
Since $\mathcal{H}_{\vartheta}(\Omega)$ is a complete space,
we obtain a unique limit element $\tilde{f} \in \mathcal{H}_{\vartheta}(\Omega)$. 
Finally we can easily show that $\tilde{f}$ coincides with $f$ by adding $0 = s_{f, X_n} - s_{f, X_n}$:
\begin{align*}
\Vert f - \tilde{f} \Vert_{L_2(\Omega)} 
&\leq \Vert f - s_{f, X_n} \Vert_{L_2(\Omega)} + \Vert \tilde{f} - s_{f, X_n} \Vert_{L_2(\Omega)} \\
&\leq c_f h_{X_n}^{\mu} + C_{\mathcal{H}_{\vartheta'}(\Omega) \hookrightarrow L_2(\Omega)} \Vert \tilde{f} - s_{f, X_n} \Vert_{\mathcal{H}_{\vartheta}(\Omega)} \stackrel{n \rightarrow \infty}{\longrightarrow} 0.
\end{align*}
Thus we have $f = \tilde{f}$ in $L_2(\Omega)$,
i.e.\ $f$ can be seen as a particular representative of $\tilde{f}$, 
such that we have proven that $f \in \mathcal{H}_{\vartheta}(\Omega)$ for any $\vartheta < \frac{\beta}{\tau}$. 

\textbf{Second case:} For $\beta \in (2\tau - \frac{d}{2}, \tau]$:
We consider the sequence of density functions $(D_{Z_n} f)_{n \in \N}$ as defined in \Cref{subsec:construction_density_func}.
We would like to show that $(TD_{Z_n})_{n \in \N}$ is a Cauchy sequence in $\mathcal{H}_\vartheta(\Omega)$ for some $\vartheta \in (2-\frac{d}{2\tau}, 2)$.
For this we consider for some $\vartheta \in (2 - \frac{d}{2\tau}, 2)$
\begin{align}
\label{eq:sum_to_consider}
\Vert TD_{Z_{n+1}} f - TD_{Z_n} f \Vert_{\mathcal{H}_{\vartheta}(\Omega)}^2 
&= \sum_{j=1}^\infty \frac{\left| \langle T(D_{Z_{n+1}}(f) - D_{Z_n}(f)), \varphi_j \rangle_{L_2(\Omega)} \right|^2}{\lambda_j^{\vartheta}} \notag \\
&= \sum_{j=1}^\infty \lambda_j^{2-\vartheta} \left| \langle D_{Z_{n+1}}(f) - D_{Z_n}(f), \varphi_j \rangle_{L_2(\Omega)} \right|^2,
\end{align}
where we used that $T: L_2(\Omega) \rightarrow L_2(\Omega)$ is self-adjoint and that $T\varphi_j = \lambda_j \varphi_j$.
Note that $\sum_{j=1}^\infty \left| \langle D_{Z_{n+1}} f - D_{Z_n} f, \varphi_j \rangle_{L_2(\Omega)} \right|^2 = \Vert D_{Z_{n+1}}(f) - D_{Z_n}(f) \Vert_{L_2(\Omega)}^2$,
for which we have only a growing upper bound in $n$ according to \Cref{prop:approx_by_TD}.
Thus the idea is to counter this growth via $\lambda_j^{2-\vartheta}$, which is a decaying term (in $j$).
We split the sum of Eq.~\eqref{eq:sum_to_consider} into two parts, small indices up to $F(n)$ (to be determined) and large indices beyond:
\begin{itemize}
\item \textbf{$j$ large}: We make use of $\lambda_{F(n)+1}^{2-\vartheta}$ being small:
\begin{align*}
&~~~ \sum_{j=F(n)+1}^\infty \lambda_j^{2-\vartheta} \left| \langle D_{Z_{n+1}}(f) - D_{Z_n}(f), \varphi_j \rangle_{L_2(\Omega)} \right|^2 \\
&\leq \lambda_{F(n)}^{2-\vartheta} \cdot \Vert D_{Z_{n+1}}(f) - D_{Z_n}(f) \Vert_{L_2(\Omega)}^2 \\
&\leq \lambda_{F(n)}^{2-\vartheta} \cdot \left( \Vert D_{Z_{n+1}}(f) \Vert_{L_2(\Omega)} + \Vert D_{Z_n}(f) \Vert_{L_2(\Omega)} \right)^2 \\
&\leq C\lambda_{F(n)}^{2-\vartheta} \cdot 2^{2n(2\tau-\beta)}.
\end{align*}
In the last step, \Cref{prop:approx_by_TD} was leveraged for bounding the $L_2(\Omega)$-norms.
\item \textbf{$j$ small}:
In this case, we make use of the lengthy but straightforward estimate provided in \Cref{subsec:estimate_for_sum},
which shows that it holds
\begin{align*}
\sum_{j=1}^{F(n)} \lambda_j^{2-\vartheta} \left| \langle D_{Z_{n+1}}(f) - D_{Z_n}(f), \varphi_j \rangle_{L_2(\Omega)} \right|^2 \leq C \lambda_{F(n)}^{-\vartheta} 2^{-2n\beta}.
\end{align*}
\end{itemize}

In view of Eq.~\eqref{eq:sum_to_consider}, and given estimates for $j$ small and $j$ big, 
we can combine both estimates and subsequently choose an optimized value of $F(n)$:
\begin{align*}
\Vert TD_{Z_{n+1}} f - TD_{Z_n} f \Vert_{\mathcal{H}_{\vartheta}(\Omega)}^2 
\leq C\lambda_{F(n)}^{2-\vartheta} \cdot 2^{n(4\tau-2\beta)} + C\lambda_{F(n)}^{-\vartheta} 2^{-2n\beta}.
\end{align*}
This expression can be minimized in $\lambda = \lambda_{F(n)}$:
\begin{align*}
\frac{\mathrm{d}}{\mathrm{d}\lambda} (\lambda^{2-\vartheta} 2^{n(4\tau-2\beta)} + \lambda^{-\vartheta} \cdot 2^{-2n\beta}) &\stackrel{!}{=} 0 \\
\Leftrightarrow ~ (2-\vartheta) \lambda^{1-\vartheta} 2^{n(4\tau-2\beta)} 
-\vartheta \lambda^{-\vartheta-1} 2^{-2n\beta} &\stackrel{!}{=} 0 \\
\Leftrightarrow \lambda_{F(n)} &= \sqrt{\frac{\vartheta}{2-\vartheta}} \cdot 2^{-2n\tau}.
\end{align*}
Since the eigenvalues $\lambda_j$ of a Sobolev kernel satisfy $\lambda_j \asymp j^{-2\tau/d}$ (see Eq.~\eqref{eq:asympt_eigvals}),
we pick $F(n) = 2^{nd}$, 
such that $\lambda_{F(n)} \asymp 2^{-2\tau n}$.
With this bound for $\lambda_{F(n)}$, we proceed as
\begin{align*}
\Vert TD_{X_{n+1}} f - TD_{X_n} f \Vert_{\mathcal{H}_{\vartheta}(\Omega)}^2 
&\leq C \left( \lambda_{F(n)}^{2-\vartheta} \cdot 2^{n(4\tau - 2\beta)} + \lambda_{F(n)}^{-\vartheta} \cdot 2^{-2n\beta} \right) \\
&\leq C \left( 2^{-2n\tau \cdot (2-\vartheta)} 2^{n(4\tau-2\beta)} + 2^{2n \tau\vartheta} 2^{-2n\beta} \right) \\
&\leq C 2^{2n(\tau \vartheta - \beta)} \\
\Rightarrow \Vert TD_{X_{n+1}} f - TD_{X_n} f \Vert_{\mathcal{H}_{\vartheta}(\Omega)} &\leq \sqrt{C} 2^{n(\tau \vartheta - \beta)}.
\end{align*} 
The exponent is negative for $\vartheta < \frac{\beta}{\tau}$.
With this estimate at hand, we can prove that $(TD_{X_n}(f))_{n \in \N}$ is a Cauchy sequence in $\mathcal{H}_\vartheta(\Omega)$ for $\vartheta < \frac{\beta}{\tau}$:
Let $n > m \geq m_0$:
\begin{align*}
\Vert TD_{X_{n}}(f) - TD_{X_m}(f) \Vert_{\mathcal{H}_\vartheta(\Omega)} 
&= 	\left \Vert \sum_{\ell=m}^{n-1} TD_{X_{\ell+1}}(f) - TD_{X_{\ell}}(f) \right \Vert_{\mathcal{H}_\vartheta(\Omega)} \\
&\leq \sum_{\ell=m}^{n-1} \left \Vert TD_{X_{\ell+1}}(f) - TD_{X_{\ell}}(f) \right \Vert_{\mathcal{H}_\vartheta(\Omega)} \\
&\leq \sum_{\ell=m}^{n-1} \sqrt{C} 2^{n(\tau \vartheta - \beta)} \\
&\leq \sqrt{C} \frac{2^{m_0(\tau\vartheta-\beta)}}{1-2^{\tau\vartheta - \beta}} \stackrel{m_0 \rightarrow \infty}{\longrightarrow} 0.
\end{align*}
Since $\mathcal{H}_\vartheta(\Omega)$ is a complete space and $(TD_{X_n}(f))_{n \in \N}$ is a Cauchy sequence in $\mathcal{H}_\vartheta(\Omega)$,
there exists a unique limiting element $\tilde{f} \in \mathcal{H}_\vartheta(\Omega)$.
Finally, $f$ coincides with $\tilde{f}$ on $\Omega$ as elements of $L_2(\Omega)$
due to \Cref{prop:approx_by_TD} and the previous calculation:
\begin{align*}
\Vert f - \tilde{f} \Vert_{L_2(\Omega)} 
&\leq \Vert f - TD_{Z_n}(f) \Vert_{L_2(\Omega)} + \Vert \tilde{f} - TD_{Z_n}(f) \Vert_{L_2(\Omega)} \\
&\leq C 2^{-n\beta} + C_{\mathcal{H}_\vartheta(\Omega) \hookrightarrow L_2(\Omega)}  \Vert \tilde{f} - TD_{Z_n}(f) \Vert_{\mathcal{H}_\vartheta(\Omega)} \stackrel{n \rightarrow \infty}{\longrightarrow} 0.
\end{align*}
Thus $f$ can be seen as a particular representative of $\tilde{f}$,
and thus we have proven that $f \in \mathcal{H}_\vartheta(\Omega)$ for all $\vartheta < \frac{\beta}{\tau}$. \\
\textbf{Third case}: For $\beta > 2\tau$,
we again consider the sequence of density functions $(D_{Z_n}(f))_{n \in \N}$.
We make use of \Cref{prop:L2_bound_density_approx_2}, i.e.\ we obtain
\begin{align*}
\Vert D_{Z_n}(f) \Vert_{L_2(\Omega)} \leq C 2^{-n(\beta-2\tau)} \stackrel{n \rightarrow \infty}{\longrightarrow} 0
\end{align*}
because of $\beta > 2\tau$.
Since $T$ is a bounded operator, 
we also have 
\begin{align*}
\Vert TD_{Z_n}(f) \Vert_{L_2(\Omega)} \stackrel{n \rightarrow \infty}{\longrightarrow} 0.
\end{align*} 
On the other hand, by \Cref{prop:approx_by_TD} we have $\Vert f - TD_{Z_n}(f) \Vert_{L_2(\Omega)} \stackrel{n \rightarrow \infty}{\longrightarrow} 0$.
Thus we obtain
\begin{align*}
\Vert f \Vert_{L_2(\Omega)} &\leq \Vert f - TD_{Z_n}(f) \Vert_{L_2(\Omega)} + \Vert TD_{Z_n}(f) \Vert_{L_2(\Omega)} \stackrel{n \rightarrow \infty}{\longrightarrow} 0,
\end{align*}
i.e.\ $f=0$.
\end{proof}

As a next step, we first present an immediate corollary of \Cref{th:main_result},
which states a saturation result using $L_p(\Omega)$, $2 \leq p \leq \infty$:

\begin{cor}
\label{cor:saturation_Lp}
Consider a compact Lipschitz region $\Omega \subset \R^d$ and a continuous kernel $k$ such that $\ns \asymp H^\tau(\Omega)$ for some $\tau > d/2$.
Consider $f \in \mathcal{C}(\Omega)$ and the estimate
\begin{align}
\label{eq:inverse_statement_assumption_Lp}
\Vert f - s_{f, X} \Vert_{L_p(\Omega)} \leq c_f h_{X}^\beta,
\end{align}
for some $p \in [2, \infty]$,
where $s_{f, X} \in \Sp \{ k(\cdot, x), x \in X\}$ is a kernel-based approximant
satisfying \Cref{ass:method}.

If Eq.~\eqref{eq:inverse_statement_assumption_Lp} holds for some $\beta > 2\tau$ 
for any quasi-uniform sequence $(X_n)_{n \in \N} \subset \Omega$, 
then $f=0$.
\end{cor}

\begin{proof}
Since $\Omega \subset \R^d$ is bounded,
Hölder inequality yields the inequality
$\Vert \cdot \Vert_{L_2(\Omega)} \leq |\Omega|^{1/2-1/p} \cdot \Vert \cdot \Vert_{L_\infty(\Omega)}$.
Hence Eq.~\eqref{eq:inverse_statement_assumption_Lp} implies Eq.~\eqref{eq:inverse_statement_assumption},
such that \Cref{th:main_result} yields the desired result $f=0$.
\end{proof}

Thus the convergence rate saturates at $2\tau$ for any $2 \leq p \leq \infty$.
Note that this improves and generalizes the previously known saturation result of \cite[Theorem 7.1]{schaback2002inverse} for thin-plate splines.

In the main result \Cref{th:main_result},
there is a slight difference in the assumptions for the regime $\beta \in [0, 2\tau - d/2)$ and $\beta \in [2\tau - d/2, 2]$.
In the first regime,
one assumes a decay rate for a single sequence $(X_n)_{n \in \N} \subset \Omega$ of quasi-uniform points,
while in the second regime one assumes a decay rate to hold for any set of quasi-uniform points.
This difference in the assumption is due to the fact that $k(\cdot, x) \in \mathcal{H}_\vartheta(\Omega)$ if and only if $\vartheta \in [0, 2-\frac{d}{2\tau})$, 
see Eq.~\eqref{eq:kernel_inclusion}:
Consider $f := k(\cdot, x_0)$,
and the point $x_0$ is included in $X_n, n \geq n_0$ within the the nested sequence $(X_n)_{n \in \N}$,
then it holds $s_{f, X_n} = f$ for all $n \geq n_0$.
Thus the approximation error $\Vert f - s_{f, X_n} \Vert_{L_2(\Omega)}$ 
will be zero, 
i.e.\ any arbitrary high approximation rate is realized.
Such an example can be easily extended to other situations,
e.g.\ where $x_0$ is actually never included in any $X_n$,
but the distance $\mathrm{dist}(x_0, X_n)$ decays with some (fast) rate.

However, this \emph{neither} implies $f \in \mathcal{H}_\vartheta(\Omega)$ for $\vartheta$ beyond $2-\frac{d}{2\tau}$ (due to Eq.~\eqref{eq:kernel_inclusion})
\emph{nor} $f=0$ (saturation).
Thus the difference in the assumptions in \Cref{th:main_result} --
for a particular sequence $(X_n)_{n \in \N}$ vs for any set of quasi-uniform points -- 
is crucial.

\subsection{Discussion of main result}
\label{subsec:discussion_main_result}

The main result \Cref{th:main_result} states an inverse result for finitely smooth Sobolev kernels with a focus on the superconvergence regime (for $\beta \in [\tau, 2\tau]$),
thus complementing recently derived corresponding direct statements \cite{karvonen2025general}.
Moreover, also the inverse statements of the escaping the native space regime (for $\beta \in (0, \tau]$) from the literature were slightly improved,
either by weaking assumptions on the point distribution \cite{wenzel2025sharp}
or by extending the class of applicable kernels \cite{avesani2025sobolev}.
Furthermore, also novel saturation statements have been derived, 
which naturally provide an upper limit on the convergence rate
and thus addresses one of the important research directions stated in \cite[Chapter 10]{buhmann2003radial}.

There are two key features about the inverse (superconvergence) statements:
First and most important, it establishes an one-to-one correspondence between smoothness (measured in terms of the power space parameter $\vartheta$) and the rate of approximation $\beta$,
thus extending the previously already established one-to-one correspondence from the escaping the native space regime \cite{wenzel2025sharp} also to the superconvergence regime:
Any continuous $f \in \mathcal{H}_\vartheta(\Omega)$ for $\vartheta \in (0, 2\tau]$ can be approximated by kernel interpolation with a rate of convergence (in the fill distance) of $h^{\vartheta \beta}$.
And if a function $f$ can be approximated by kernel interpolation with a rate of convergence of $h^{\vartheta \beta}$, 
one can conclude that $f \in \mathcal{H}_{\vartheta'}(\Omega)$ for all $\vartheta' < \vartheta$.
We will address implications and applications of this one-to-one correspondence for the challenging selection of an ``optimal shape parameter'' for RBF kernels in an follow up work \cite{wenzel2026optimal}.
Second, the result shows that the power spaces $\mathcal{H}_\vartheta(\Omega)$ for $\vartheta \in [1, 2]$ are the correct spaces to expect superconvergence to happen for.
This is important, because the previous literature on superconvergence \cite{schaback2018superconvergence,santin2022sampling,sloan2025doubling,karvonen2025general}
was able to show superconvergence for specific subspaces,
that satisfy certain embeddings.
However the relation between these different characterizations remained unclear.
The inverse result \Cref{th:main_result} now implies,
that these specific subspaces are always subspaces of some power space $\mathcal{H}_\vartheta(\Omega)$.

Next we disscuss the assumptions of the main result \Cref{th:main_result}:
The assumption on the set $\Omega$ of interest being a Lipschitz region is pretty standard,
and the assumption on the kernel being a Sobolev kernel covers plenty of kernels that are used in practice,
including Wendland and Matérn kernels.
The assumption on the set of points -- 
either a single nested sequence of quasi-uniform points vs for any set of quasi-uniform points -- was mostly already discussed below \Cref{th:main_result}.
Instead of assuming a decay rate in the $L_2(\Omega)$ norm, 
one could also either consider some $\mathcal{H}_\vartheta(\Omega)$ norm (smoothness scale),
some $L_p(\Omega)$ norm (integrability scale) 
or a combination of both.
Regarding the first option, this is easily possible by combining the main result \Cref{th:main_result} with the Bernstein inequalities provided in \Cref{th:bernstein}.
The second option, and thus also the third one,
is possible in the range $p \in [2, \infty]$,
due to the embeddings $L_p(\Omega) \hookrightarrow L_2(\Omega)$ (see e.g.\ \Cref{cor:saturation_Lp}) and the availability of corresponding direct statements \cite{wendland2005approximate}.
The range $p \in [1, 2]$ is more challenging.
Nevertheless the $\ell_p-L_p$ Bernstein inequalities established in \cite[Theorem 5.3]{ward2012lp} (and similar \cite{mhaskar2010bernstein} for the sphere) may provide helpful tools to generalize the analysis provided here for $p=2$ to the more general cases,
and thus completing the full picture for $TL_p(\Omega)$ spaces as provided in \cite[Figure 2]{karvonen2025general}.
We remark that for $p \in [1, 2]$ one likely needs to resort to Besov spaces instead of Hilbertian power spaces.

Finally we comment on the role of the approximant $s_{f, X}$:
On purpose, the main result \Cref{th:main_result} was formulated for general point based approximants $s_{f, X} \in \Sp \{ k(\cdot, x), x \in X \}$ (partly assuming the mild condition of \Cref{ass:method}),
not only focussing on kernel interpolants.
This allows to apply the result not only to kernel interpolation,
but also to other approximation schemes like regularized interpolation, 
least-squares approximation or kernel Galerkin methods.
Note that the corresponding direct statements are not yet available,
as the direct superconvergence statements in \cite{karvonen2025general} relied explicitly on the use of projections (see e.g.\ \cite[Theorem 6]{karvonen2025general}).

Direct and inverse statements on manifold are also well researched on the sphere,
and e.g.\ \cite{fuselier2012scattered} provides direct statements in the superconvergence regime.
However corresponding inverse statements as well as saturation statements seem to be missing.
We expect that our proof techniques of \Cref{th:main_result} can also be transfered to the case of $\Omega$ being a sphere or more generally a manifold.
For the special case of the sphere,
the analysis might be even more simple due to the absense of a boundary and the explicit availability of spherical basis functions.

Finally we observe the following implication of the main result \Cref{th:main_result} and the previously established direct and inverse statements:
For $\vartheta \in [0, 2]$,
the  kernel-based approximation $s_{f, X_n}$ of $f \in \mathcal{H}_\vartheta(\Omega)$ using well-distributed points $X_n \subset \Omega$ provides the same rate of approximation as approximation with the Mercer eigenfunctions.
For $\vartheta > 2$,
the kernel-based approximation $s_{f, X_n}$ can no longer match the rate of approximation using eigenfunctions,
due to the saturation established in \Cref{th:main_result} and \Cref{cor:saturation_Lp}.
Let $V_n := \Sp \{ \varphi_j, j=1, ..., n\}$ and recall the classical error estimate using eigenfunctions
\begin{align*}
\Vert f - \Pi_{V_n}(f) \Vert_{L_2(\Omega)}^2 &= \sum_{j=N+1}^\infty \lambda_j^\vartheta \cdot \frac{|\langle f, \varphi_j \rangle_{L_2(\Omega)}|^2}{\lambda_j^\vartheta} \leq \lambda_{N+1}^\vartheta \cdot \Vert f \Vert_{L_2(\Omega)}^2.
\end{align*}
Due to $\lambda_N \asymp N^{-2\tau/d}$ (see Eq.~\eqref{eq:asympt_eigvals}) and $h_X \asymp |X|^{-1/d}$ for well distributed $X \subset \Omega$,
this results in the same error rates as using kernel-based approximants, 
see Eq.~\eqref{eq:error_bound_escaping} and Eq.~\eqref{eq:error_bound_superconv}.
However, this approximation rate using Mercer eigenfunctions is not limited by $\vartheta \leq 2$, as for kernel-based approximation due to the saturation result in \Cref{th:main_result} and \Cref{cor:saturation_Lp}.
On the other hand, the kernel-based approximation is actually computable, 
whereas the Mercer eigenfunctions are usually unknown.

\section{Conclusion}
\label{sec:conclusion}

This paper proved sharp inverse and saturation statements for kernel-based approximation using finitely smooth kernels with a focus on the superconvergence regime.
Thus this analysis complements recently derived direct statements for the superconvergence regime.
This was enabled by deriving a range of important utility statements,
such as a generalized reproducing property and a superconvergence Bernstein inequality,
and furthermore proposing and anlyzing a density function construction.
Together with direct and inverse statments for the escaping the native space regime,
a complete characterization of the convergence rates was derived,
giving a one-to-one correspondence between convergence rate and smoothness (measured in terms of power spaces).

In future research we will discuss implications and applications of the established one-to-one correspondence.
In particular, we will address the prominent question of the optimal shape parameter from a theoretical point of view \cite{wenzel2026optimal}.

~ \\ 

\textbf{Acknowledgements:}
The author likes to thank Daniel Winkle and Gabriele Santin for helpful discussions.

\addtocontents{toc}{\protect\setcounter{tocdepth}{-10}} %
\IfFileExists{/home/wenzel/references.bib}{
\bibliography{/home/wenzel/references}				%
}{
\IfFileExists{/home/ians1/wenzeltn/references.bib}{
\bibliography{/home/ians1/wenzeltn/references}      %
}{
\bibliography{/home/wenzeltn/references}			%
}
}
\bibliographystyle{abbrv}

\appendix

\section{Outsourced calculations}

\subsection{Bernstein inequality in the escaping regime}
\label{subsec:bernstein_ineq_proof}

\begin{proof}[Proof of \Cref{prop:escaping_bernstein}]
The proof follows by realizing that \cite[Lemma 4.1]{zhengjie2025inverse} actually holds for any $u \in H^m(\Omega)$,
and \cite[Lemma 4.3]{zhengjie2025inverse} holds as long as the RKHS of the kernel is norm-equivalent to $H^\tau(\Omega)$, 
since we make use of the fact that $u$ interpolates $f_\sigma$ in $X$.
We provide the details:

We start by noting that \cite[Lemma 4.1]{zhengjie2025inverse} actually holds for any $u \in H^m(\Omega)$,
i.e.\ no kernel is involved in this lemma yet.
Next we observe, that \cite[Lemma 4.3]{zhengjie2025inverse} also works for any Sobolev kernel,
i.e.\ for any $\alpha \in (d/2, \tau]$ there exists a constant $C > 0$ such that it holds
\begin{align}
\label{eq:intermediate_eq_2}
\Vert u \Vert_{H^\tau(\Omega)} \leq C q_X^{-\tau + \alpha} \Vert u \Vert_{H^\alpha(\Omega)}
\end{align}
for any trial functions $u \in \Sp \{ k(\cdot, x), x \in X\}$:
In fact, applying \cite[Lemma 4.1]{zhengjie2025inverse} to $u \in \Sp \{ k(\cdot, x), x \in X \} \subset \ns \asymp H^\tau(\Omega) \subseteq H^\alpha(\Omega)$ yields a band-limited function $f_{\sigma, u, \alpha, \Omega} \in \mathcal{B}_\sigma$ for some $\sigma \asymp q_X^{-1}$
such that
\begin{align*}
u|_X &= f_{\sigma, u, \alpha, \Omega}|_X, \\
\Vert f_{\sigma, u, \alpha, \Omega} \Vert_{H^\alpha(\R^d)} &\leq C_{\alpha, \tau, \Omega} \Vert u \Vert_{H^\alpha(\Omega)}, \\
\Vert u - f_{\sigma, u, \alpha, \Omega} \Vert_{H^\alpha(\Omega)} &\leq C'_{\alpha, m, \Omega} q_X^{\tau-\alpha} \Vert u \Vert_{H^m(\Omega)},
\end{align*}
and any $f_\sigma \in \mathcal{B}_\sigma$ satisfies the Bernstein inequality
\begin{align*}
\Vert f_\sigma \Vert_{H^\tau(\Omega)} \leq 2^{(\tau-\beta)/2} \max \{ 1, \sigma^{\tau-\beta} \} \Vert f_\sigma \Vert_{H^\beta(\R^d)}.
\end{align*}
Thus we may estimate by noting that $u = s_{f_{\sigma, u, \alpha, \Omega}, X}$,
\begin{align*}
\Vert u \Vert_{H^\tau(\Omega)} 
&\leq \Vert u - f_{\sigma, u, \alpha, \Omega} \Vert_{H^\tau(\Omega)} + \Vert f_{\sigma, u, \alpha, \Omega} \Vert_{H^\tau(\Omega)} \\
&= \Vert s_{f_{\sigma, u, \alpha, \Omega}, X} - f_\sigma \Vert_{H^\tau(\Omega)} + \Vert f_\sigma \Vert_{H^\tau(\Omega)} \\
&\leq C \Vert f_{\sigma, u, \alpha, \Omega} \Vert_{H^\tau(\Omega)}.
\end{align*}
In the last inequality, we used that $s_{f_{\sigma, u, \alpha, \Omega}, X}$ is the orthogonal projection of $f_{\sigma, u, \alpha, \Omega}$ onto $\Sp \{ k(\cdot, x), x \in X \} \subset \ns \asymp H^\tau(\Omega)$.
Using aboves inequalities due to \cite[Lemma 4.1]{zhengjie2025inverse},
we can further estimate
\begin{align*}
\Vert f_{\sigma, u, \alpha, \Omega} \Vert_{H^\tau(\Omega)}
\leq \Vert f_{\sigma, u, \alpha, \Omega} \Vert_{H^\tau(\R^d)}
&\leq 2^{(\tau-\beta)/2} \max \{ 1, \sigma^{\tau-\beta} \} \Vert f_{\sigma, u, \alpha, \Omega} \Vert_{H^\beta(\R^d)} \\
&\leq 2^{(\tau-\beta)/2} \max \{ 1, \sigma^{\tau-\beta} \} \Vert f_{\sigma, u, \alpha, \Omega} \Vert_{H^\beta(\Omega)} \\
&\leq C q_X^{-\tau+\beta} \Vert f_{\sigma, u, \alpha, \Omega} \Vert_{H^\beta(\Omega)},
\end{align*}
establishing Eq.~\eqref{eq:intermediate_eq_2}. \\
We continue to estimate the right hand side norm $\Vert u \Vert_{H^\alpha(\Omega)}$ of Eq.~\eqref{eq:intermediate_eq_2} via
Gangliardo-Nierenberg interpolation inequality as $\Vert u \Vert_{H^\alpha(\Omega)} \leq C \Vert u \Vert_{L_2(\Omega)}^{1-\alpha/\tau} \Vert u \Vert_{H^\tau(\Omega)}^{\alpha/\tau}$.
Rearranging the terms yields
\begin{align*}
\Vert u \Vert_{H^\tau(\Omega)} \leq C q_X^{-\tau} \Vert u \Vert_{L_2(\Omega)}.
\end{align*}
Another application of Gangliardo-Nierenberg interpolation inequality as in \cite[Theorem 3.4]{avesani2025sobolev} yields the final result:
\begin{align*}
\Vert u \Vert_{H^{\vartheta \tau}(\Omega)} 
&\leq  C \Vert u \Vert_{L_2(\Omega)}^{1-\vartheta} \Vert u \Vert_{H^{\tau}(\Omega)}^{\vartheta} \\
&\leq C \Vert u \Vert_{L_2(\Omega)}^{1-\vartheta} \left( C q_X^{-\tau} \Vert u \Vert_{L_2(\Omega)} \right)^{\vartheta} \\
&= C q_X^{–\vartheta\tau} \Vert u \Vert_{L_2(\Omega)}.
\end{align*}
\end{proof}

\subsection{Estimate for sum}
\label{subsec:estimate_for_sum}

\begin{prop}
In the setting of \Cref{th:main_result} it holds
\begin{align*}
\sum_{j=1}^{F(n)} \lambda_j^{2-\vartheta} \left| \langle D_{Z_{n+1}}(f) - D_{Z_n}(f), \varphi_j \rangle_{L_2(\Omega)} \right|^2 \leq C \lambda_{F(n)}^{-\vartheta} 2^{-2n\beta}.
\end{align*}

\end{prop}

\begin{proof}
We start with basic computations, using the definition of $D_{Z_n}$ as introduced in \Cref{subsec:construction_density_func}:
\begin{align*}
&~~~\left| \langle D_{Z_{n+1}}(f) - D_{Z_n}(f), \varphi_j \rangle_{L_2(\Omega)} \right| \\
&= \left| \int_{\Omega} (D_{Z_{n+1}}(f) - D_{Z_n}(f))(x) \varphi_j(x) ~ \mathrm{d}x \right|.
\end{align*}
Note that both $D_{Z_n}(f)$ and $D_{Z_{n+1}}(f)$ are zero on $\Omega \setminus \left( Z_{n+1} + [0, 2^{-(n+1)}]^d \right)$ by definition,
such that we can restrict the integral to $Z_{n+1} + [0, 2^{-(n+1)}]^d$.
We continue as
\begin{align*}
&~~~\left| \langle D_{Z_{n+1}}(f) - D_{Z_n}(f), \varphi_j \rangle_{L_2(\Omega)} \right| \\
&= \Big| \sum_{z \in Z_{n+1}} \int_{z + [0, 2^{-(n+1)}]^d} (D_{Z_{n+1}}(f) - D_{Z_n}(f))(x) \varphi_j(x) ~ \mathrm{d}x \\
&= \Big| \sum_{z \in Z_{n+1}} \int_{[0, 2^{-(n+1)}]^d} (D_{Z_{n+1}}(f) - D_{Z_n}(f))(z+b) \langle k(\cdot, z+b), \varphi_j \rangle_{\ns} ~ \mathrm{d}b \Big| \\
&= \frac{1}{\lambda_j} \cdot \Big| \sum_{z \in Z_{n+1}} \int_{[0, 2^{-(n+1)}]^d} \left \langle (D_{Z_{n+1}}(f) - D_{Z_n}(f))(z+b) k(\cdot, z+b), \varphi_j \right \rangle_{L_2(\Omega)} \Big|.
\end{align*}
In the penultimate step, we made use of the reproducing property,
and in the last step we made use of $\varphi_j = \frac{1}{\lambda_j} T\varphi_j$ and subsequently Eq.~\eqref{eq:important_identity}.
As a next step, we write the integral as a limit of Riemann sums to obtain
\begin{align*}
&= \frac{1}{\lambda_j} \cdot \Big| \sum_{z \in Z_{n+1}} \lim_{N \rightarrow \infty} \frac{2^{-(n+1)d}}{N^d} \sum_{b \in \frac{2^{-(n+1)}}{N} \{ 0, ..., N-1 \}^d } \left \langle (D_{Z_{n+1}}(f) - D_{Z_n}(f))(z+b) k(\cdot, z+b), \varphi_j \right \rangle_{L_2(\Omega)} \Big| \\ 
&= \frac{1}{\lambda_j} \cdot \Big| \lim_{N \rightarrow \infty} \frac{2^{-(n+1)d}}{N^d} \sum_{b \in \frac{2^{-(n+1)}}{N} \{ 0, ..., N-1 \}^d } \left \langle \sum_{z \in Z_n} (D_{Z_{n+1}}(f) - D_{Z_n}(f))(z+b) k(\cdot, z+b), \varphi_j \right \rangle_{L_2(\Omega)} \Big| \\
&\leq \frac{1}{\lambda_j} \lim_{N \rightarrow \infty} \frac{2^{-(n+1)d}}{N^d} \sum_{b \in \frac{2^{-(n+1)}}{N} \{ 0, ..., N-1 \}^d } \Big| \left \langle \sum_{z \in Z_n} (D_{Z_{n+1}}(f) - D_{Z_n}(f))(z+b) k(\cdot, z+b), \varphi_j \right \rangle_{L_2(\Omega)} \Big|.
\end{align*}
From this we conclude by using Cauchy Schwarz inequality
\begin{align*}
&~~~\left| \langle D_{Z_{n+1}}(f) - D_{Z_n}(f), \varphi_j \rangle_{L_2(\Omega)} \right|^2 \\
&\leq \frac{1}{\lambda_j^2} \lim_{N \rightarrow \infty} \frac{2^{-2(n+1)d}}{N^{2d}} \left( \sum_{b \in \frac{2^{-(n+1)}}{N} \{ 0, ..., N-1 \}^d } \Big| \left \langle \sum_{z \in Z_{n+1}} (D_{Z_{n+1}}(f) - D_{Z_n}(f))(z+b) k(\cdot, z+b), \varphi_j \right \rangle_{L_2(\Omega)} \Big| \right)^2 \\
&= \frac{1}{\lambda_j^2} \lim_{N \rightarrow \infty} \frac{2^{-2(n+1)d}}{N^{2d}} \left( \sum_{b \in \frac{2^{-(n+1)}}{N} \{ 0, ..., N-1 \}^d } 1^2 \right) \\
&\qquad \quad \cdot \left( \sum_{b \in \frac{2^{-(n+1)}}{N} \{ 0, ..., N-1 \}^d} \Big| \left \langle \sum_{z \in Z_{n+1}} (D_{Z_{n+1}}(f) - D_{Z_n}(f))(z+b) k(\cdot, z+b), \varphi_j \right \rangle_{L_2(\Omega)} \Big|^2 \right) \\ 
&= \frac{1}{\lambda_j^2} \lim_{N \rightarrow \infty} \frac{2^{-2(n+1)d}}{N^{d}} \sum_{b \in \frac{2^{-(n+1)}}{N} \{ 0, ..., N-1 \}^d } \left| \left \langle \sum_{z \in Z_{n+1}} (D_{Z_{n+1}}(f) - D_{Z_n}(f))(z+b) k(\cdot, z+b), \varphi_j \right \rangle_{L_2(\Omega)} \right|^2 \\
&\leq \frac{1}{\lambda_j^2} \lim_{N \rightarrow \infty} \frac{2^{-2(n+1)d}}{N^{d}} \sum_{b \in \frac{2^{-(n+1)}}{N} \{ 0, ..., N-1 \}^d } \left| \left \langle \sum_{z \in Z_{n+1}} (D_{Z_{n+1}}(f) - D_{Z_n}(f))(z+b) k(\cdot, z+b), \varphi_j \right \rangle_{L_2(\Omega)} \right|^2.
\end{align*}
Thus we obtain
\begin{align*}
&~~~ \sum_{j=1}^{F(n)} \lambda_j^{2-\vartheta} \left| \langle D_{Z_{n+1}}(f) - D_{Z_n}(f), \varphi_j \rangle_{L_2(\Omega)} \right|^2 \\
&\leq \sum_{j=1}^{F(n)} \lambda_j^{-\vartheta} \lim_{N \rightarrow \infty} \frac{2^{-2(n+1)d}}{N^{d}} \sum_{b \in \frac{2^{-(n+1)}}{N} \{ 0, ..., N-1 \}^d } \left| \left \langle \sum_{z \in Z_{n+1}} (D_{Z_{n+1}}(f) - D_{Z_n}(f))(z+b) k(\cdot, z+b), \varphi_j \right \rangle_{L_2(\Omega)} \right|^2 \\
&\leq \lambda_{F(n)}^{-\vartheta} \lim_{N \rightarrow \infty} \frac{2^{-2(n+1)d}}{N^{d}} \sum_{b \in \frac{2^{-(n+1)}}{N} \{ 0, ..., N-1 \}^d } \sum_{j=1}^{F(n)} \left| \left \langle \sum_{z \in Z_{n+1}} (D_{Z_{n+1}}(f) - D_{Z_n}(f))(z+b) k(\cdot, z+b), \varphi_j \right \rangle_{L_2(\Omega)} \right|^2 \\
&\leq \lambda_{F(n)}^{-\vartheta} \lim_{N \rightarrow \infty} \frac{2^{-2(n+1)d}}{N^{d}} \sum_{b \in \frac{2^{-(n+1)}}{N} \{ 0, ..., N-1 \}^d } \left \Vert \sum_{z \in Z_{n+1}} (D_{Z_{n+1}}(f) - D_{Z_n}(f))(z+b) k(\cdot, z+b) \right \Vert_{L_2(\Omega)}^2
\end{align*}
The $\Vert \cdot \Vert_{L_2(\Omega)}$-norm can be estimated by inserting the definition of $D_{Z_{n+1}}(f)$ and $D_{Z_n}(f)$:
\begin{align*}
&~\left \Vert \sum_{z \in Z_{n+1}} (D_{Z_{n+1}}(f) - D_{Z_n}(f))(z+b) k(\cdot, z+b) \right \Vert_{L_2(\Omega)} \\
=&~\left \Vert \sum_{z \in Z_{n+1}} D_{Z_{n+1}}(f)(z+b)k(\cdot, z+b) - D_{Z_n}(f)(z+b) k(\cdot, z+b) \right \Vert_{L_2(\Omega)} \\
=&~\left \Vert \sum_{z \in Z_{n+1}} D_{Z_{n+1}}(f)(z+b)k(\cdot, z+b) 
- \sum_{b_i \in B_{n+1}} \sum_{z \in Z_n} D_{Z_n}(f)(z+b_i+b) k(\cdot, z+b_i+b) \right \Vert_{L_2(\Omega)} \\
=&~\left \Vert \sum_{z \in Z_{n+1}} 2^{(n+1)d} \alpha_{f, Z_{n+1}+b; z+b} k(\cdot, z+b) 
- \sum_{b_i \in B_{n+1}} \sum_{z \in Z_{n}} 2^{nd} \alpha_{f, Z_n+b_i+b; z+b_i+b} k(\cdot, z+b_i+b) \right \Vert_{L_2(\Omega)}
\end{align*}
\begin{align*}
=&~ 2^{nd} \left \Vert \sum_{z \in Z_{n+1}} 2^{d} \alpha_{f, Z_{n+1}+b; z+b} k(\cdot, z+b) 
- \sum_{b_i \in B_{n+1}} \sum_{z \in Z_{n}} \alpha_{f, Z_n+b_i+b; z+b_i+b} k(\cdot, z+b_i+b) \right \Vert_{L_2(\Omega)} \\
=&~ 2^{nd} \left \Vert \sum_{b_i \in B_{n+1}} \sum_{z \in Z_{n+1}} \alpha_{f, Z_{n+1}+b; z+b} k(\cdot, z+b) 
- \sum_{b_i \in B_{n+1}} \sum_{z \in Z_{n}} \alpha_{f, Z_n+b_i+b; z+b_i+b} k(\cdot, z+b_i+b) \right \Vert_{L_2(\Omega)} \\
=&~ 2^{nd} \left \Vert \sum_{b_i \in B_{n+1}} \left( \sum_{z \in Z_{n+1}} \alpha_{f, Z_{n+1}+b; z+b} k(\cdot, z+b) 
- \sum_{z \in Z_{n}} \alpha_{f, Z_n+b_i+b; z+b_i+b} k(\cdot, z+b_i+b) \right) \right \Vert_{L_2(\Omega)} \\
=&~ 2^{nd} \left \Vert \sum_{b_i \in B_{n+1}} \left( s_{f, Z_{n+1}+b}
- s_{f, Z_n + b_i + b} \right) \right \Vert_{L_2(\Omega)} \\
=&~ 2^{nd} \sum_{b_i \in B_{n+1}} \left \Vert s_{f, Z_{n+1}+b}
- s_{f, Z_n + b_i + b} \right \Vert_{L_2(\Omega)} \\
\leq&~ 2^{nd} \sum_{b_i \in B_{n+1}} \left( \left \Vert s_{f, Z_{n+1}+b} - f \Vert_{L_2(\Omega)} + \Vert f - s_{f, Z_n + b_i + b} \right \Vert_{L_2(\Omega)} \right).
\end{align*}
For this expression, we can finally apply the given convergence rate of Eq.~\eqref{eq:inverse_statement_assumption} and the estimates on the fill distance of \Cref{prop:estimate_sep_fill_dist_Z} to bound it via
\begin{align*}
\leq&~ 2^{nd} \sum_{b_i \in B_{n+1}} \left( c_f h_{Z_{n+1}+b}^\beta + c_f h_{Z_n+b_i+b}^\beta \right) \\
\leq&~ 2^{nd} \sum_{b_i \in B_{n+1}} \left( c_f C_\Omega^\beta 2^{-(n+1)\beta} + c_f C_\Omega^\beta 2^{-n\beta} \right) \\
\leq&~ 2 c_f C_\Omega^\beta 2^d \cdot 2^{nd} 2^{-n\beta}.
\end{align*}
Plugging this estimate into aboves calculation, we obtain
\begin{align*}
&~~~ \sum_{j=1}^{F(n)} \lambda_j^{2-\vartheta} \left| \langle D_{Z_{n+1}}(f) - D_{Z_n}(f), \varphi_j \rangle_{L_2(\Omega)} \right|^2 \\
&\leq \lambda_{F(n)}^{-\vartheta} \lim_{N \rightarrow \infty} \frac{2^{-2(n+1)d}}{N^{d}} \sum_{b \in \frac{2^{-(n+1)}}{N} \{ 0, ..., N-1 \}^d } C 2^{2nd} 2^{-2n\beta} \\
&\leq C\lambda_{F(n)}^{-\vartheta} \lim_{N \rightarrow \infty} \frac{1}{N^{d}} \sum_{b \in \frac{2^{-(n+1)}}{N} \{ 0, ..., N-1 \}^d } 2^{-2n\beta} \\
&= C \lambda_{F(n)}^{-\vartheta} 2^{-2n\beta}.
\end{align*}
\end{proof}

\end{document}